\documentclass[12pt]{extarticle}
\usepackage{amsmath, amsthm, amssymb, mathtools, hyperref,color}
\usepackage{graphicx}
\usepackage{caption}
\usepackage{subcaption}
\usepackage{tikz-cd}

\newtheorem{theorem}{Theorem}
\numberwithin{theorem}{section}
\newtheorem{proposition}[theorem]{Proposition}
\newtheorem{lemma}[theorem]{Lemma}
\newtheorem{corollary}[theorem]{Corollary}
\newtheorem{definition}[theorem]{Definition}
\newtheorem{remark}[theorem]{Remark}

\newtheorem{example}[theorem]{Example}

\begin{document}
\title{Tutte Short Exact Sequences of Graphs}
\author{Madhusudan Manjunath \footnote {Part of this work was carried out while the author was visiting MFO, Oberwolfach, IHES, Bures-sur-Yvette and ICTP, Trieste. We thank the generous support and the warm hospitality of these institutes. The author was supported by a MATRICS grant of the Department of Science and Technology (DST), India.}}
\maketitle

\begin{abstract}
We associate two modules, the $G$-parking critical module and the toppling critical module, to an undirected connected graph $G$. The $G$-parking critical module and the toppling critical module are canonical modules (with suitable twists) of quotient rings of the well-studied $G$-parking function ideal and the toppling ideal, respectively.  For each critical module, we establish a Tutte-like short exact sequence relating the modules associated to $G$, an edge contraction $G/e$ and an edge deletion $G \setminus e$ ($e$ is a non-bridge). We obtain purely combinatorial consequences of Tutte short exact sequences. For instance, we reprove a theorem of Merino that the critical polynomial of a graph is an evaluation of its Tutte polynomial, and relate the vanishing of certain combinatorial invariants (the number of acyclic orientations on connected partition graphs satisfying a unique sink property) of $G/e$ to the equality of the corresponding invariants of $G$ and $G \setminus e$. 
\end{abstract}

%%%%Definition of critical modules: as Bayer-Sturmfels modules, as canonical modules of the toppling ideals.

\section{Introduction}

Let $G$ be an undirected, connected, multigraph on $n$-vertices labelled $v_1,\dots,v_n$ and with $\ell$ loops. Let $\mathbb{K}$ be a field and $R=\mathbb{K}[x_1,\dots,x_n]$ be the polynomial ring in $n$ variables with coefficients in $\mathbb{K}$. The toppling ideal $I_G$ of $R$ is a binomial ideal that encodes chip firing equivalence on $G$ \cite{CorRosSal00} and the $G$-parking function ideal $M_G$ is a monomial initial ideal of $I_G$ that closely mirrors the properties of $I_G$ \cite{PosSha04}. The ideals $I_G$ and $M_G$ (and their quotient rings) have received significant attention recently, in part due to their connections with tropical geometry. Combinatorial commutative algebraic aspects of $R/I_G$ and $R/M_G$ such as their minimal free resolutions in terms of the underlying graph $G$ have been studied from several perspectives \cite{PerPerWil11, ManStu13, ManSchWil15, DocSan14, MohSho16}.

Both $R/I_G$ and $R/M_G$ are Cohen-Macaulay (of depth and Krull dimension one) and hence, have associated canonical modules (also known as dualising modules)  \cite[Part b, Proposition 3.6.12]{BruHer98}. We refer to the canonical modules of $R/I_G$ and $R/M_G$ (both twisted by the number of loops of $G$) as the \emph{toppling critical module} and the \emph{$G$-parking critical module}, respectively. We denote the toppling critical module and the $G$-parking critical module by ${\rm CTopp}_G$ and ${\rm CPark}_G$, respectively. 
  In this article, we posit that the critical modules behave better compared to the corresponding quotient rings in certain contexts. Specifically, we construct short exact sequences relating the critical modules of $G$, its contraction $G/e$ and deletion $G \setminus e$ by an edge $e$ that is not a bridge {\footnote{This condition ensures that $G \setminus e$ remains connected and is needed to guarantee that resulting coordinate ring retains some basic properties. For instance, if $G$ is not connected, then the quotient ring of the toppling ideal defined analogously does not have Krull dimension one.}}. 
Taking cue from the deletion-contraction sequence that characterises the Tutte polynomial of a graph, we refer to these sequences as \emph{Tutte short exact sequences}.  

We present purely combinatorial consequences of Tutte short exact sequences. For instance,  as a corollary we obtain an algebraic proof of a theorem of Merino \cite{Mer97} that the critical polynomial of a graph is an evaluation of its Tutte polynomial. 
This follows from the additivity of the Hilbert series of the modules involved in one of the Tutte short exact sequences, namely the $G$-parking Tutte short exact sequence. By considering associated long exact sequences of ${\rm Tor}$,  we relate the vanishing of certain combinatorial invariants  of $G/e$ to the equality of corresponding invariants of $G$ and $G \setminus e$. These combinatorial invariants are the number of acyclic orientations satisfying a unique sink property on certain graphs derived from $G$ called \emph{connected partition graphs} \cite[Pages 2854--2855]{ManSchWil15}. We also note a deletion-contraction formula for certain numbers associated to $G$ called \emph{alternating numbers} that are alternating sums of these combinatorial invariants. 

 The construction of the Tutte short exact sequences and the corresponding proofs involve a delicate interplay between the algebraic structure of the critical modules and the combinatorial structure of the graph, mainly its acyclic orientations.  In the following, we state our main theorems concerning Tutte short exact sequences.  Before  this, we clarify one crucial point about contraction and deletion of the edge $e$.

 {\bf The Notions $G/e$, $G \setminus e$ and $G/(v_i,v_j)$:} Suppose that there are  $m_e \geq 1$ edges between $v_1$ and $v_2$. By $G/e$, we mean the graph obtained from $G$ by contracting the vertices $v_1$ and $v_2$  to the vertex $v_{1,2}$ and with $m_e-1$ loops on the vertex $v_{1,2}$.  By $G \setminus e$, we mean the graph obtained from $G$ by deleting the edge $e$ and retaining all the other $m_e-1$ edges parallel to $e$. On the other hand, by $G/(v_i,v_j)$ for a pair of distinct, adjacent vertices $(v_i,v_j)$ we mean the graph obtained by contracting every edge between $v_i$ and $v_j$.

\subsection{Tutte Short Exact Sequences}

Let $e$ be an edge between the vertices $v_1$ and $v_2$ that is not a bridge.  Let $R_e$ be the polynomial ring $\mathbb{K}[x_{1,2},x_3,\dots,x_n]$ in $(n-1)$-variables with coefficients in $\mathbb{K}$ so that its variables are naturally in correspondence with the vertices of $G/e$.

{\bf$G$-Parking Tutte Short Exact Sequence}: We construct a short exact sequence relating the $G$-parking critical modules of $G$, its contraction $G/e$ and deletion $G \setminus e$ with respect to the edge $e$. By definition, the $G$-parking critical modules of $G$ and $G \setminus e$ are $R$-modules, whereas the $G$-parking critical module of $G/e$ is an $R_e$-module.  We start by realising ${\rm CPark}_G$ and ${\rm CPark}_{G \setminus e}$ as $R_e$-modules. For this, we consider the linear form $L:=x_1-x_2$ and note that $R_e \cong R/\langle L \rangle$ via a map between $R$ and $R_e$ that takes $x_1$ and $x_2$ to $x_{1,2}$ and $x_i$ to itself for all $i \neq 1,2$. This isomorphism realises $R_e$ as an $R$-module.  We consider the tensor products ${\rm CPark}_G \otimes_R R_e$ and ${\rm CPark}_{G \setminus e} \otimes_R R_e$  as $R_e$-modules.  We define $R_e$-module maps $\psi_0: {\rm CPark}_{G/e} \rightarrow {\rm CPark}_G \otimes_R R_e$ and $\phi_0: {\rm CPark}_G \otimes_R R_e \rightarrow {\rm CPark}_{G \setminus e} \otimes_R R_e$.  We denote the map ${\rm CPark}_{G/e}/{\rm ker}(\psi_0) \rightarrow {\rm CPark}_G \otimes_R R_e$ induced by $\psi_0$ also by $\psi_0$. We show that $\psi_0$ and $\phi_0$ fit into a short exact sequence. More precisely,  

\begin{theorem}{{\rm{\bf{($G$-parking Tutte Short Exact Sequence)}}}}\label{GparkTutte_theo}
Let $G$ be an undirected connected multigraph (possibly with loops) with at least three vertices. Let $e$ be an edge between the vertices $v_1$ and $v_2$ that is not a bridge.  
The kernel of the map $\psi_0$ is equal to $x_{1,2} \cdot {\rm CPark}_{G/e}$ and the following sequence of $R_e$-modules: 

\begin{center}
$0 \rightarrow {\rm CPark}_{G/e}/(x_{1,2} \cdot {\rm CPark}_{G/e}) \xrightarrow{\psi_0} {\rm CPark}_G \otimes_R R_e \xrightarrow{\phi_0} {\rm CPark}_{G \setminus e} \otimes_R R_e \rightarrow 0$
 \end{center}
 is a short exact sequence of graded $R_e$-modules.
\end{theorem}

% In following, we construct a short exact sequence relating the critical modules of $G$, $G/e$ and $G$ %\setminus e$ for an edge $e$  {\color{blue} that is neither a bridge nor a multiple edge}.  L
%{\color{blue} Discuss $G/e$ with more care: when $e$ is a multiple edge}
{\bf Toppling Tutte Short Exact Sequence}: The toppling critical module ${\rm CTopp}_{G/e}$ is by definition an $R_e$-module (rather than an $R$-module). In contrast, ${\rm CTopp}_G$ and ${\rm CTopp}_{G \setminus e}$ are by definition $R$-modules. We start by realising ${\rm CTopp}_G$ and ${\rm CTopp}_{G \setminus e}$ as $R_e$-modules. For this, we realise $R_e$ as an $R$-module via the same isomorphism $R_e \cong R/\langle L \rangle$ as in the $G$-parking case and regard ${\rm CTopp}_G \otimes_R R_e$ and ${\rm CTopp}_{G \setminus e} \otimes_R R_e$ as $R_e$-modules. Hence, ${\rm CTopp}_G \otimes_R R_e$ and ${\rm CTopp}_{G \setminus e} \otimes_R R_e$  are $R_e$-modules. We define $R_e$-module maps $\psi_1: {\rm CTopp}_{G/e} \rightarrow {\rm CTopp}_G \otimes_R R_e$ and $\phi_1: {\rm CTopp}_G \otimes_R R_e \rightarrow {\rm CTopp}_{G \setminus e} \otimes_R R_e$. We also denote by $\psi_1$,  the injective map ${\rm CTopp}_{G/e}/{\rm ker}(\psi_1) \rightarrow {\rm CTopp}_{G} \otimes_R R_e$ induced by $\psi_1$.

\begin{theorem} {\rm ({\bf Toppling Tutte Short Exact Sequence})}\label{Tutteses_theo}
Let $\mathbb{K}$ be a field of characteristic two. Let $G$ be an undirected connected multigraph (possibly with loops) with at least three vertices. Let $e$ be an edge in $G$ between $v_1$ and $v_2$ that is not a bridge.  The following sequence of $R_e$-modules:

\begin{equation}
 0 \rightarrow {\rm CTopp}_{G/e}/{\rm ker}(\psi_1) \xrightarrow{\psi_1} {\rm CTopp}_G \otimes_R R_e \xrightarrow{\phi_1} {\rm CTopp}_{G \setminus e} \otimes_R R_e \rightarrow 0
\end{equation}
is a short exact sequence of graded $R_e$-modules.  

%{ \color{blue} Furthermore, this short exact sequence splits and hence, we have the following isomorphism of graded $R_e$-modules:
%\begin{center}
%$C_G \otimes_R R_e \cong (C_{G/e}/{\rm ker}(\psi_1)) \oplus (C_{G \setminus e} \otimes_R R_e)$
%\end{center}
%}
\end{theorem}

%{\color{blue} Discuss the split exactness of each of these sequences}

\begin{remark}
\rm
 The  $G$-parking Tutte short exact sequence is not split exact in general. To see this, suppose that $e$ has parallel edges then both ${\rm CPark}_G \otimes_R R_e$ and ${\rm CPark}_{G \setminus e} \otimes_R R_e$ have the same number of minimal generators. If the corresponding Tutte short exact sequence was split exact, then we would have $\beta_0({\rm CPark}_{G}\otimes_R R_e)=\beta_0({\rm CPark}_{G \setminus e}\otimes_R R_e)+\beta_0({\rm CPark}_{G/e})$ which is not true.
 We do not know whether the toppling short exact sequence is split exact.  
 %We note however that the obvious candidate section map from ${\rm GC}_{G/e} \otimes_R R_e$ to ${\rm GC}_G \otimes_R R_e$ that takes $[\mathcal{A''}]$ to $[\mathcal{A''}_{e^{+}}]$ (where for the acyclic orientation $\mathcal{A''}$ on $G \setminus e$, the acyclic orientation $\mathcal{A''}_{e^{+}}$ on $G$ is obtained by further orienting $e$ so that $v_1$ is the source)  is not well-defined. 
\qed
\end{remark}

\begin{remark}
\rm 
Note that unlike the case of the $G$-parking critical module, $(x_1-x_2)$ is never a non-zero divisor of ${\rm CTopp}_G$ for any connected graph $G$. This can be seen by showing the equivalent property that $(x_1-x_2)$ is a zero divisor of $R/I_G$ which in turn follows from the fact that $G$ has a principal divisor, see \cite[Page 768]{BakNor07} for the definition, of the form $d \cdot (v_1)-d \cdot (v_2)$ for some positive integer $d$.  The kernel of $\psi_1$ is also in general more complicated in this case (see the last line of Example \ref{triangle_ex}):  in general, it only strictly contains $x_{1,2} \cdot {\rm CTopp}_{G/e}$.  \qed

\end{remark}

\begin{remark}
\rm
We expect that Theorem \ref{Tutteses_theo} does not require characteristic two and believe that it can be generalised to arbitrary ground fields. We rely on characteristic two in, for instance, the proof of Proposition \ref{toppzeroh1_prop}.
\qed
\end{remark}
% We decided to restrict to characteristic two in the body of the paper since it suffices for the combinatorial applications and the main ideas of the proof are better exposed in this case. 

%{\color{blue} just the assumption is not a multiple edge}

\subsection{Motivation and Applications}  Two sources of motivation for the Tutte short exact sequence are i. Merino's theorem \cite{Mer97} and its connection to Stanley's O-sequence conjecture \cite{Mer01}, ii. divisor theory on graphs \cite{BakNor07}. 

Merino's theorem  states that the generating function of the critical configurations of $G$ is an evaluation of the Tutte polynomial at $(1,t)$.  The first observation that relates the critical modules to Merino's theorem is that their Hilbert series are both equal to $P_G(t)/(1-t)$, where $P_G(t)$ is the generating function of the critical configurations of $G$ (this is implicit in \cite{ManStu13}, also see Remark \ref{hileq_rem}). This leads to the question of whether Merino's theorem can be enriched into a short exact sequence of critical modules. Merino's theorem can then be recovered from this short exact sequence from the fact that the Hilbert series is additive in short exact sequences. Such a short exact sequence might then allow the possibility of obtaining further combinatorial results by, for instance, considering the associated long exact sequence of Tor, Ext and other derived functors. The $G$-parking Tutte short exact sequence is such an enrichment and can be viewed as a categorification of Merino's theorem.  By studying the associated long exact sequence in Tor, we relate certain combinatorial invariants of $G/e$ to those of $G$ and $G \setminus e$. We refer to \cite{Hel-Ron05} and \cite{Jas-Ron06} for a categorification of the chromatic polynomial  and the Tutte polynomial of a graph, respectively.  These works seem to be a different flavour from the current work. 

Merino's theorem is a key ingredient in the proof of Stanley's $O$-sequence conjecture for co-graphic matroids \cite{Mer01}. Stanley's conjecture is still open for arbitrary matroids. We raise the question of exploring generalisations of the main results of this paper to matroids as a possible approach to Stanley's conjecture.

{\bf Merino's Theorem via Tutte Short Exact Sequences:} As an application of the $G$-parking Tutte short exact sequence, we deduce the following version of Merino's theorem as a corollary to Theorem \ref{GparkTutte_theo}. 
Recall that the $K$-polynomial of a finitely generated graded module over the (graded) polynomial ring is the numerator of the Hilbert series expressed as a rational function in reduced form. 

\begin{theorem} {\rm({\bf Merino's Theorem})}
The $K$-polynomial of ${\rm CPark}_G$ is the Tutte evaluation $T_G(1,t)$, where $T_G(x,y)$ is the Tutte polynomial of $G$. 
\end{theorem} 

Next, we note a deletion-contraction formula for alternating sums of the graded Betti numbers $\beta_{i,j}$ which is an immediate consequence of Merino's theorem but does not seem to appear in literature. 

{\bf A Deletion-Contraction Formula for Alternating Numbers:}  For an integer $k$, let $\mathcal{A}_{k}=\sum_{i} (-1)^i \beta_{i,k}$ be the $k$-th alternating number of $H$.   We have the following deletion-contraction formula for the numbers $\mathcal{A}_k$:

\begin{proposition}{\rm {\bf (Deletion-Contraction for Alternating Numbers)}}\label{delete-contra_prop}
%Let $G$ be an undirected connected multigraph (possibly with loops) and let $e$ be an edge of $G$ that is not a bridge. 

The numbers $\mathcal{A}_k(G)$ satisfy the following formula: 

\begin{equation}\label{alternate_form}
\mathcal{A}_k(G)+\mathcal{A}_{k-1}(G/e)=\mathcal{A}_{k}(G/e)+\mathcal{A}_k(G \setminus e).
\end{equation}
\end{proposition}

\begin{example}
\rm
Suppose $G$ is a triangle with vertices $v_1,v_2$ and $v_3$ and let $e=(v_1,v_2)$. The associated numbers are the following. 
\begin{center}
 $\mathcal{A}_k(G)=\mathcal{A}_k(G/e)=\mathcal{A}_k(G \setminus e)=0$ for $k<0$.
 $\mathcal{A}_0(G)=2, \mathcal{A}_1(G)=-3,\mathcal{A}_2(G)=0,\mathcal{A}_3(G)=1,~\mathcal{A}_k(G)=0$ for $k \geq 4,$.
 $\mathcal{A}_0(G/e)=1,\mathcal{A}_1(G/e)=0,\mathcal{A}_2(G/e)=-1,~\mathcal{A}_k(G/e)=0$ for $k \geq 3$.
 $\mathcal{A}_0(G \setminus e)=1,\mathcal{A}_1(G \setminus e)=-2,\mathcal{A}_2(G \setminus e)=1,~\mathcal{A}_k(G \setminus e)=0$ for $k \geq 3$.
 \end{center}
Note that Formula (\ref{alternate_form}) is satisfied for various values of $k$.
 \qed
\end{example}

Note that $\mathcal{A}_0(H)$ is the number of acyclic orientations on $H$ with a unique sink at $v_2$ and $\mathcal{A}_{-1}(H)=0$. Hence, as a corollary we obtain the familiar formula:

\begin{center}
$\mathcal{A}_0(G)=\mathcal{A}_0(G/e)+\mathcal{A}_0(G \setminus e)$.
\end{center}

{\bf Equality of Betti numbers of $G$ and $G \setminus e$ in terms of Vanishing of Betti numbers of $G/e$:}  Let $H$ be an undirected, connected, multigraph with $n$ vertices, $m$ edges and $\ell$ loops.  Following \cite[Page 2854]{ManSchWil15}, a connected $i$-partition of $H$ is a partition  $\Pi=\{V_1,\dots,V_i\}$  of its vertex set of size $i$ such that the subgraph induced by each subset is connected.  The \emph{connected partition graph} associated to this partition $\Pi$ is the multigraph with $\Pi$ as its vertex set and with $\hat{a}_{i,j}$ edges between $V_i$ and $V_j$, where  $a_{u,v}$ is the number of edges between vertices $u$ and $v$ of $G$ and $\hat{a}_{i,j}=\sum_{u \in V_i, v\in V_j}a_{u,v}$. We define $\beta_{i,j+\ell}(H)$ to be the number of acyclic orientations on connected partition graphs of $H$ of size $n-i$, with $m-j$ edges and with a unique sink at the partition containing $v_2$ (or any other fixed vertex).  
Note that from \cite{ManSchWil15} and the graded version of \cite[Corollary 3.3.9]{BruHer98}, we know that these are the graded Betti numbers of both ${\rm CPark}_H$ and ${\rm CTopp}_H$ (see Proposition \ref{critbetti_prop} ). 

\begin{theorem}\label{vanimpeq_theo}
Let $G$ be an undirected connected graph (with possible loops) and let $e$ be an edge of $G$ that is not a bridge. For any $(i,j) \in \mathbb{Z}^2$, if $\beta_{i,j}(G/e)=\beta_{i-1,j-1}(G/e)=\beta_{i-1,j}(G/e)=\beta_{i-2,j-1}(G/e)=0$, then $\beta_{i,j}(G)=\beta_{i,j}(G \setminus e)$.
\end{theorem}

\begin{example}
\rm
Suppose that $G$ is a triangle with vertices $v_1,v_2$ and $v_3$ and let $e=(v_1,v_2)$.  The Betti numbers are the following:

\begin{center}
$\beta_{0,0}(G)=2, \beta_{1,1}(G)=3,\beta_{2,3}(G)=1$,\\
$\beta_{0,0}(G/e)=1,\beta_{1,2}(G/e)=1$,\\
$\beta_{0,0}(G \setminus e)=1,\beta_{1,1}(G \setminus e)=2, \beta_{2,2}(G \setminus e)=1$.
\end{center}

For $(i,j)=(2,4)$ the hypothesis of Theorem \ref{vanimpeq_theo} are all satisfied and we have $\beta_{2,4}(G)=\beta_{2,4}(G \setminus e)=0$.  We currently do not know of examples where the hypothesis of Theorem \ref{vanimpeq_theo} are satisfied and $\beta_{i,j}(G)=\beta_{i,j}(G \setminus e) \neq 0$.
\qed 
\end{example}

At the time of writing, we do not know of a combinatorial proof of Theorem \ref{vanimpeq_theo}.

%%%%Define the graded Betti numbers of $\beta_{i,j}(G)$ of ${\rm GC}_G$ (and $C_G$).
%%%%Define their alternating sums $\mathcal{A}(G)$ over the homological degree.
%%%%State the formula in terms of these alternating sums.
%%%Recover the case of acylic orientations with a unique sink from this. 

%We refer to Section {\color{blue} give reference} for the long exact sequence of Tor associated to the Tutte short exact sequence and {\color{blue} its combinatorial meaning}.   

{\bf Connections to Divisor Theory and Related Sequences:} The toppling critical module has an interpretation in terms of divisor theory of graphs. This connection is implicit in \cite{ManStu13}. The punchline is that the Hilbert coefficients of the toppling critical module ${\rm CTopp}_G$ count linear equivalence classes of divisors $D$ whose rank of $D$ is equal to the degree of $D$ minus $g$, where $g$ is the genus of the graph (recall that $g=m-n+1$, where $m,~n$ are the number of edges and vertices, respectively of $G$).  It seems plausible that the toppling Tutte short exact sequence also has analogous combinatorial applications: one difficulty in this direction seems to be that the kernel of the map $\psi_1$ does not seem to have a simple description.

Short exact sequences in the same spirit as the Tutte short exact sequences have appeared in literature.  For instance,  \cite[Proposition 3.4]{OrlTer92}  construct a deletion-restriction short exact sequence of Orlik-Solomon algebras of (central) hyperplane arrangements.  We leave the question of relating the deletion-restriction short exact sequence associated to the graphical arrangement to the Tutte short exact sequences in this paper as a topic for further work.  In a related direction, Dochtermann and Sanyal \cite{DocSan14} use the graphical hyperplane arrangement to compute the minimal free resolution of the $G$-parking function ideal. This work has been extended to the toppling ideal by Mohammadi and Shokrieh \cite{MohSho16}.

\section{The Maps and Proof Sketch} 

  In this section, we describe the maps $\psi_i,\phi_i$ and sketch the proofs of Theorem \ref{GparkTutte_theo} and Theorem \ref{Tutteses_theo}. The maps arise naturally from the combinatorial interpretation of the minimal generators of the (toppling and $G$-parking) critical modules. 

A key input to this is the combinatorial description of the minimal generators and the first syzygies, i.e., a generating set for the relations between the minimal generators, of the critical modules implicit in \cite{ManSchWil15}.  We summarise this description here. 

The minimal generators of ${\rm CPark}_G$ are in bijection with acyclic orientations on $G$ with a unique sink at $v_2$.  The  minimal generators of ${\rm CTopp}_G$ are in bijection with equivalence classes of acyclic orientations on $G$ defined as follows \cite{BakNor07}. 

Given an acyclic orientation $\mathcal{A}$ on $G$,  consider the divisor: \begin{center} $D_{\mathcal{A}}=\sum_{v}({\rm outdeg}_{\mathcal{A}}(v)-1)(v)$,\end{center} where ${\rm outdeg}_{\mathcal{A}}(v)$ is the outdegree of $v$ with respect to the acyclic orientation $\mathcal{A}$. Define an equivalence class on the set of acyclic orientations on $G$ by declaring two acyclic orientations as equivalent if their associated divisors are linearly equivalent \cite[Section 1.6]{BakNor07}.  Given an acyclic orientation on $\mathcal{A}$, we denote its equivalence class by $[\mathcal{A}]$. Once a vertex $v_2$ say is fixed, $[\mathcal{A}]$ has a canonical representative: the acyclic orientation with a unique sink at $v_2$ that is equivalent to $\mathcal{A}$. Such an acyclic orientation exists and is unique \cite[Section 3.1]{BakNor07}.  Hence, the two critical modules have the same number of minimal generators. We refer to these generating sets as the \emph{standard generating sets}. Furthermore, by the right exactness of the tensor product functor they induce a generating set on the $R_e$-modules ${\rm CPark}_G\otimes_R R_e$ and ${\rm CTopp}_G \otimes_R R_e$ that we also refer to as the \emph{standard generating sets}.

The first syzygies of the critical modules have (minimal) generators that correspond to certain acyclic orientations on graphs obtained by contracting a pair of vertices that are connected by an edge.   We refer to these  as the \emph{standard syzygies}. Also, by the right exactness of the tensor product functor, they induce a generating set of the first syzygies of the corresponding $R_e$-modules ${\rm CPark}_G\otimes_R R_e$ and ${\rm CTopp}_G \otimes_R R_e$ that we refer to by the same terminology.  We refer to Subsection \ref{freepres_subsect} for more details. 

{\bf The Maps $\psi_0$ and $\phi_0$:}  We use the free presentation described above to define $\psi_0$ and $\phi_0$.   Let $m_e$ be the multiplicity of the edge $e$.   The map $\psi_0$ takes the minimal generator $\mathcal{A}$ on $G/e$ corresponding to an acyclic orientation with a unique sink at $v_{1,2}$ to $x_{1,2}^{m_e-1}\mathcal{A}_{e^{+}} \in {\rm CPark}_{G} \otimes_R R_e$, where $\mathcal{A}_{e^{+}}$ is the minimal generator corresponding to the acyclic orientation obtained by further orienting $e$ such that $v_1$ is the source of $e$. We identify this minimal generator with the corresponding acyclic orientation.  Note that the resulting acyclic orientation also has a unique sink at $v_2$.

We turn to the definition of $\phi_0$. We distinguish between two cases: $m_e=1$ and $m_e>1$. 

 Consider the case where $m_e=1$.  Suppose $\mathcal{A'}$ is an acyclic orientation on $G$ with a unique sink at $v_2$, following \cite{ManSchWil15} we say that an edge of $G$ is contractible on  $\mathcal{A'}$ if the orientation $\mathcal{A'}/e$ induced by $\mathcal{A'}$ on $G/e$ is acyclic.  If $e$ is not contractible on $\mathcal{A'}$, then $v_1$ must be a source of at least one edge on $\mathcal{A'}/e$ and hence, $\mathcal{A'} \setminus e$ has a unique sink at $v_2$.  The map $\phi_0$ is defined as follows:

 %We denote by $\mathcal{A''}$ the unique acyclic orientation on $G \setminus e$ with a unique sink at $v_2$ that is equivalent to the acyclic orientation $\mathcal{A'} \setminus e$ on $G \setminus e$. {\color{blue} is this necessary?} 
 
\[
\phi_0(\mathcal{A'})=
\begin{cases}
\mathcal{A'}\setminus e,\text{ if the edge $e$ is not contractible on $\mathcal{A'}$},\\
0, \text{ otherwise.}
\end{cases}
\]

Suppose that $m_e>1$.  We define $\phi_0(\mathcal{A'})=\mathcal{A'} \setminus e$ for every standard generator $\mathcal{A}'$ of ${\rm CPark}_{G} \otimes_R R_e$. Note that apriori the maps $\psi_0$ and $\phi_0$  are only candidate maps and their well-definedness needs further argumentation.  We will carry this out in Section \ref{gparktutteses_Sect}.

{\bf The Maps $\psi_1$ and $\phi_1$:}  Suppose that $\mathcal{A}$ is an acyclic orientation on $G/e$. Let $\mathcal{A}_{e^{+}}$  and $\mathcal{A}_{e^{-}}$ be acyclic orientations on $G$ obtained by further orienting $e=(v_1,v_2)$ such that $v_1$ and $v_2$ is the source of $e$, respectively. If $e$ is a simple edge, then the map $\psi_1$ takes the generator $[\mathcal{A}]$ of ${\rm CTopp}_{G/e}$ corresponding to the class of $\mathcal{A}$ to  $[\mathcal{A}_{e^{+}}]+[\mathcal{A}_{e^{-}}]$ in ${\rm CTopp}_{G} \otimes_R R_e$. More generally, if $e$ is an edge of multiplicity $m_e$, then the map $\psi_1$ takes the generator $[\mathcal{A}]$ of ${\rm CTopp}_{G/e} \otimes_R R_e$ to  $x_{1,2}^{m_e-1}[\mathcal{A}_{e^{+}}]+x_{1,2}^{m_e-1}[\mathcal{A}_{e^{-}}]$ in ${\rm CTopp}_{G} \otimes_R R_e$. 
%where $x_1$ and $x_2$ are the sources of $e$ in $\mathcal{A}_{e^{+}}$ and $\mathcal{A}_{e^{-}}$ respectively.  

Suppose that $\mathcal{A'}$ is an acyclic orientation on $G$, let $\mathcal{A'}\setminus e$ be the acyclic orientation on $G \setminus e$ induced by $\mathcal{A'}$ i.e., by deleting the edge $e$. The map $\phi_1$ takes the generator $[\mathcal{A'}]$ of ${\rm CTopp}_G \otimes_R R_e$ to $[\mathcal{A'} \setminus e]$ in ${\rm CTopp}_{G \setminus e} \otimes_R R_e$.

Note that the fact that the maps $\psi_1$ and $\phi_1$ are well-defined requires proof. The proof of well-definedness consists of two parts: i. showing that the maps do not depend on the choice of representatives of the classes $[\mathcal{A}]$ and $[\mathcal{A'}]$, ii.   showing that they induce $R_e$-module maps $\psi_1: {\rm CTopp}_{G/e} \rightarrow  {\rm CTopp}_G \otimes_R R_e$ and $\phi_1: {\rm CTopp}_G \otimes_R R_e \rightarrow {\rm CTopp}_{G \setminus e} \otimes_R R_e$.

Next, we outline the proofs of Theorem \ref{GparkTutte_theo}  and Theorem \ref{Tutteses_theo}.  A philosophy that is adopted in both these proofs is the following: ``the critical module associated to $G$ has the same structure as those associated to both $G/e$ and $G \setminus e$ except that the contraction and deletion operations respectively modify them slightly and the maps $\psi_i$ and $\phi_i$ (for $i=0$ and $1$) capture this modification". 
Both the proofs consist of the following two parts.
\begin{enumerate}

\item {\bf The Complex property:} In this step, we show that the sequence of modules in  Theorem \ref{GparkTutte_theo} and Theorem \ref{Tutteses_theo} is a complex of $R_e$-modules. To this end, we verify that the image of $\psi_i$ is contained in the kernel of $\phi_i$.  

\item {\bf The Homology of the Tutte complex:}  We show that the homology of the $G$-parking and the toppling Tutte complex  is zero at every homological degree.  In both cases, the argument is straightforward in homological degrees zero and two. 
%At homological degrees zero and two, this property follows from the injectivity of the induced map $\psi_1: C_{G/e}/(x_{1,2} \cdot C_{G/e}) \rightarrow C_G$ and surjectivity of $\phi_1$ respectively. The first statement follows from the computation of the kernel of $\psi_1$ and the surjectivity of $\phi_1$ is immediate from its construction. 

The argument is more involved at homological degree one: we must show that kernel of $\phi_i$ is equal to the image of $\psi_i$.  In order to give a flavour of the argument, we outline the argument for the toppling Tutte complex. The overall strategy is the same for the $G$-parking Tutte complex. 

The key step is to explicitly compute the kernel of $\phi_1$. We show that ${\rm ker}(\phi_1)=\{ x_{1,2}^{m_e-1}[\mathcal{A}_{e^{+}}]+x_{1,2}^{m_e-1}[\mathcal{A}_{e^{-}}]|~$ over all acyclic orientations $\mathcal{A}$ on $G/e\}$.  
For this, we use the combinatorial description of the syzygies of the toppling critical module from Subsection \ref{freepres_subsect}. The basic idea is as follows:  Suppose $\alpha \in {\rm ker}(\phi_1)$ and  that  $\alpha=\sum_{[\mathcal{A}]} p_{[\mathcal{A}]} \cdot [\mathcal{A}]$  in terms of the standard generating set of ${\rm CTopp}_G \otimes_R R_e$.  Since $\alpha \in {\rm ker}(\phi_1)$ we know that  $\sum_{[\mathcal{A}]} p_{[\mathcal{A}]} \cdot \phi_1([\mathcal{A}])=0$ and gives a syzygy of ${\rm CTopp}_{G \setminus e} \otimes_R R_e$. Hence, this syzygy can be written as an $R_e$-linear combination of the standard syzygies of ${\rm CTopp}_{G \setminus e} \otimes_R R_e$. 
Next, comparing the standard syzygies of ${\rm CTopp}_G \otimes_R R_e$ and ${\rm CTopp}_{G \setminus e} \otimes_R R_e$, we conclude that $\alpha$ is generated by the elements $x_{1,2}^{m_e-1}[\mathcal{A}_{e^{+}}]+x_{1,2}^{m_e-1}[\mathcal{A}_{e^{-}}]$ in ${\rm CTopp}_G \otimes_R R_e$.  The key idea behind comparing the standard syzygies of ${\rm CTopp}_G \otimes_R R_e$ and ${\rm CTopp}_{G \setminus e} \otimes_R R_e$ is that upon deleting the edge $e$, all the standard syzygies of ${\rm CTopp}_G \otimes_R R_e$ except the ones corresponding to contracting the edge $e$ carry over to ${\rm CTopp}_{G \setminus e} \otimes_R R_e$.  We refer to Proposition \ref{toppzeroh1_prop} for more details.

\end{enumerate}

\begin{remark}
\rm
We use the terminology $\psi_0,\psi_1$ and $\phi_0,\phi_1$ to reflect the fact by using the family $I_{G,t}$ from \cite{ManSchWil15} we can define a one parameter family of critical $R$-modules $C_{G,t}$ such that $C_{G,0}={\rm CPark}_G$ and $C_{G,1}={\rm CTopp}_G$.  It seems plausible that there is a Tutte short exact sequence for the critical module $C_{G,t}$ that interpolates between the two Tutte short exact sequences constructed here. The corresponding maps $\psi_t$ and $\phi_t$ seem more involved and we leave this for future work. \qed 
\end{remark}

% takes the minimal generator $\mathcal{A'}$  on $G$ to the minimal generator  $\mathcal{A}''$ corresponding to the unique acyclic orientation on $G \setminus e$ with a unique sink at $v_2$ and equivalent to the acyclic orientation $\mathcal{A'} \setminus e$ on $G \setminus e$ obtained from $\mathcal{A'}$ by deleting $e$. Note that $\mathcal{A'} \setminus e$ is acyclic but need not have a unique sink. 

\begin{example}{\rm{ \bf (Triangle)}}\label{triangle_ex}
\rm
Consider the case where $G=K_3$ a complete graph on three vertices labelled $v_1,v_2,v_3$ and $e=(v_1,v_2)$.  The graph $G/e$ is a multigraph on two vertices $(v_{1,2},v_3)$ with two multiple edges and $G \setminus e$ is a tree on three vertices with edges $(v_1,v_3)$ and $(v_2,v_3)$.

\begin{figure}
\centering
  \includegraphics[height=12cm]{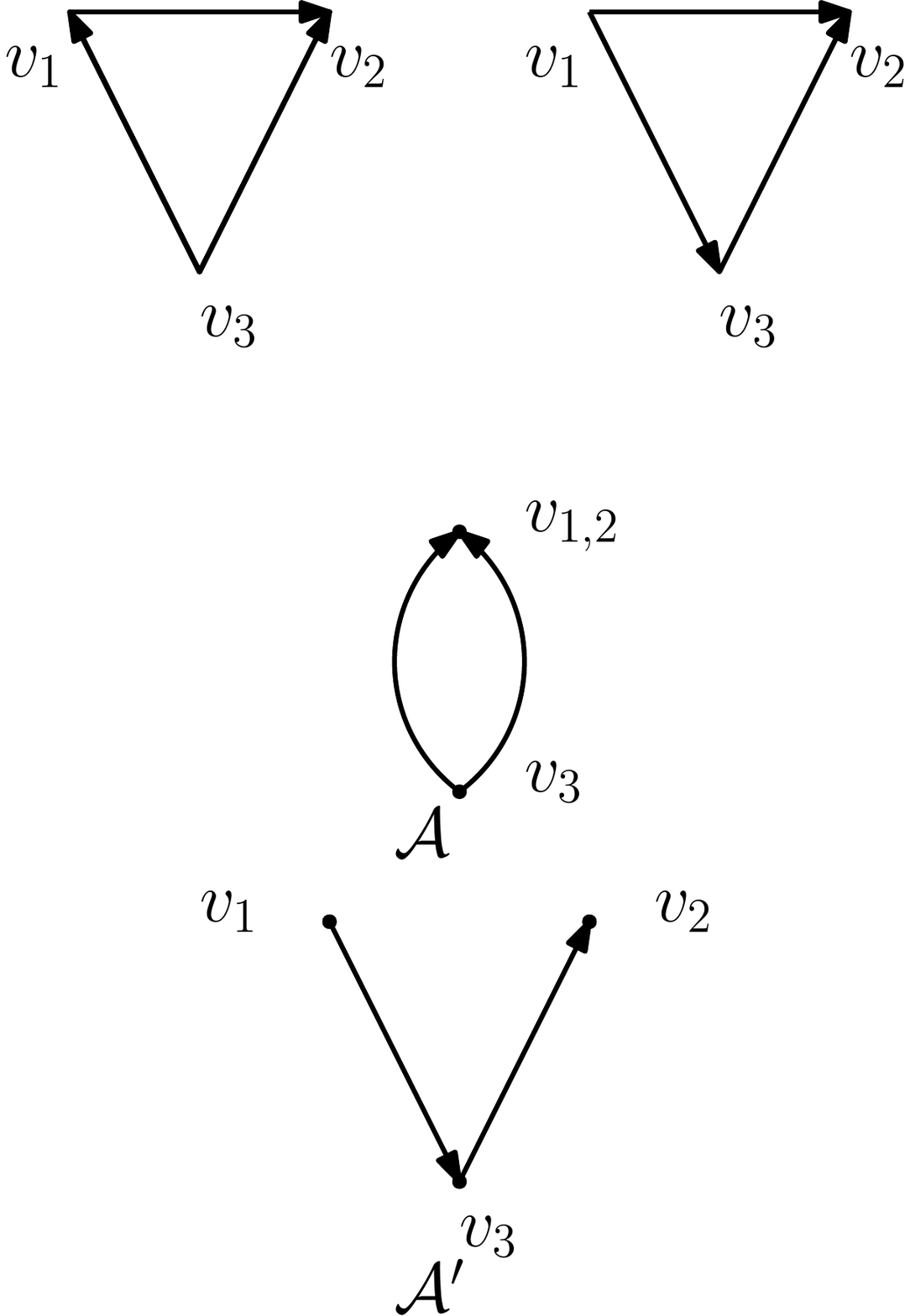}
  \caption{Acyclic Orientations Corresponding to the Minimal Generators of the Critical Modules: the case where $G=K_3$. The acyclic orientations on the top left and top right are $\mathcal{A}_{e^{+}}$ and $\mathcal{A}_{e^{-}}$ respectively.}\label{ex1_fig}
\end{figure}

The $G$-parking critical module ${\rm CPark}_G$ is generated by two elements $\mathcal{A}_{e^{+}}$ and $\mathcal{A}_{e^{-}}$ labelled by acyclic orientations shown in Figure \ref{ex1_fig} with the following relations (see Subsection \ref{freepresgparkcan_sect} for more details): 

\begin{center}
 $x_1 \cdot \mathcal{A}_{e^{+}}=0,$\\
 $x_3 \cdot \mathcal{A}_{e^{-}}=0,$\\
 $x_3 \cdot \mathcal{A}_{e^{+}}+x_1 \cdot \mathcal{A}_{e^{-}}=0$.
 \end{center}
 
 Note that $R_e/M_{G/e}$ and $R/M_{G \setminus e}$ are Gorenstein.  The  $G$-parking critical module ${\rm CPark}_{G/e}$ is generated by one element $\mathcal{A}$ labelled by the acyclic orientation with sink at $v_{1,2}$ shown in Figure \ref{ex1_fig}  with the relation:
 \begin{center}
 $x_3^2 \cdot \mathcal{A}=0$.
 \end{center}
 The $G$-parking critical module ${\rm CPark}_{G \setminus e}$ is also generated by one element $\mathcal{A'}$ labelled by the acyclic orientation with unique sink at $v_2$ shown in Figure \ref{ex1_fig} subject to relations:
\begin{center} 
 $x_1 \cdot \mathcal{A'}=0$,\\
 $x_3  \cdot\mathcal{A'}=0$.
 \end{center}
 The map $\psi_0$ takes $\mathcal{A}$ to $\mathcal{A}_{e^{+}}$ and is well-defined since $x_3^2 \cdot \psi_0(\mathcal{A})=x_3^2 \cdot \mathcal{A}_{e^{+}}=0$. Note that this relation  can be obtained from the defining relations of ${\rm CPark}_G \otimes_R R_e$ as $-x_{1,2}(x_3 \cdot \mathcal{A}_{{e}^{-}})+x_3(x_3 \cdot \mathcal{A}_{e^{+}}+x_{1,2} \cdot \mathcal{A}_{e^{-}})$.
 
 The map $\phi_0$ takes $\mathcal{A}_{e^{+}}$ to zero and $\mathcal{A}_{e^{-}}$ to $\mathcal{A'}$ and is well-defined (preserves relations).  
 
 Furthermore, the sequence is a complex since $\phi_0(\psi_0(\mathcal{A}))=0$,  the kernel of $\psi_0$ contains $x_{1,2} \cdot {\rm CPark}_{G/e}$ and the map $\phi_0$ is surjective.
 
 Next, we give a flavour of the argument for short exactness. At the homological degree one, the element $x_{1,2} \cdot \mathcal{A}_{{e}^{-}}$ is in the kernel of $\phi_0$. However, it is also in the image of $\psi_0$ since from the third defining relation of ${\rm CPark}_G$ we have $x_{1,2} \cdot \mathcal{A}_{{e}^{-}}=-x_3 \cdot \mathcal{A}_{e^{+}}=-x_3 \cdot \psi_0(\mathcal{A})$ and is hence, in the image of $\psi_0$. We generalise this argument in Section \ref{gparktutteses_Sect}.  Furthermore,  the kernel of the map $\psi_0$ turns to be precisely $x_{1,2} \cdot {\rm CPark}_{G/e}$.  
 
 Merino's theorem follows by noting that $x_1-x_2$ is a regular element on ${\rm CPark}_G$ and ${\rm CPark}_{G \setminus e}$ and $x_{1,2}$ is a regular element on ${\rm CPark}_{G/e}$ and from the additivity of the Hilbert series in short exact sequences.
 The Hilbert series of ${\rm CPark}_G,{\rm CPark}_{G/e}$ and ${\rm CPark}_{G \setminus e}$ are $(2+t)/(1-t),(1+t)/(1-t)$ and $1/(1-t)$. Hence, the Hilbert series of ${\rm CPark}_G \otimes_R R_e,{\rm CPark}_{G/e}/(x_{1,2} \cdot {\rm CPark}_{G/e})$ and ${\rm CPark}_{G \setminus e} \otimes_R R_e$ are $2+t,1+t$ and $1$, respectively.

 The toppling critical module of $G$ is also generated by two elements $[\mathcal{A}_{e^{+}}]$ and $[\mathcal{A}_{e^{-}}]$ that naturally correspond to equivalence classes of acyclic orientations $\mathcal{A}_{e^{+}}$ and $\mathcal{A}_{e^{-}}$ shown in Figure \ref{ex1_fig}, with the following relations ((see Subsection \ref{freepres_subsect} for more details)):
 
 \begin{center}  
 $x_1 \cdot [\mathcal{A}_{e^{+}}]+x_2 \cdot [\mathcal{A}_{e^{-}}]=0,$\\
 $x_2 \cdot [\mathcal{A}_{e^{+}}]+x_3 \cdot [\mathcal{A}_{e^{-}}]=0,$\\
 $x_3 \cdot [\mathcal{A}_{e^{+}}]+x_1 \cdot [\mathcal{A}_{e^{-}}]=0$.
 \end{center}

 The toppling critical modules of $G/e$ is generated by one element $[\mathcal{A}]$ labelled by the equivalence class of the acyclic orientation $\mathcal{A}$ with sink at $v_{1,2}$, as shown in Figure \ref{ex1_fig}, with the relation:

\begin{center}
 $(x_3^2+x_{1,2}^2) \cdot [\mathcal{A}]=0$.
 \end{center}

 The toppling critical module of $G \setminus e$ is generated by one element $[\mathcal{A'}]$ labelled by the equivalence class of the acyclic orientation $\mathcal{A'}$, as shown in Figure \ref{ex1_fig}, and with the relations:
 
 \begin{center}
 $(x_1+x_3) \cdot [\mathcal{A'}]=0,$\\
  $(x_3+x_2) \cdot [\mathcal{A'}]=0.$
\end{center}

Note that ${\rm CTopp}_{G \setminus e}$ is isomorphic to $R/\langle x_1+x_3,x_3+x_2 \rangle$. The map $\psi_1$ takes $[\mathcal{A}]$ to $[\mathcal{A}_{e^{+}}]+[\mathcal{A}_{e^{-}}]$ and the map $\phi_1$ takes both $[\mathcal{A}_{e^{+}}]$ and $[\mathcal{A}_{e^{-}}]$ to $[\mathcal{A'}]$. A quick check shows that these candidate maps are indeed well-defined.  Furthermore, note that the element $x_3 \cdot [\mathcal{A}]$ is in the kernel of $\psi_1$ and is not contained in $x_{1,2} \cdot {\rm CTopp}_{G/e}$.  \qed
\end{example}

{\bf Acknowledgements:} We thank Spencer Backman for several illuminating discussions that helped shape this work.  We thank the anonymous referees for their valuable suggestions.  We thank Jugal Verma for sharing his knowledge of canonical modules. 

%%%%The long exact sequence of Tor. 

%%{\color{blue} Statement of the Main Theorem}

%sthe toppling ideal is a binomial ideal generated by binomials of the form ${\bf x^u}-{\bf x^v}$ where ${\bf u},~{\bf v} \in \mathbb{Z}_{\geq 0}^{n}$ and ${\bf u}-{\bf v} \in L_G$. 

%%%%Define toppling Ideals and its canonical module. 

%%%Statement of the Main Theorem. 

%%%Applications.

%%%%Key Steps of the Proof: 

\section{Preliminaries}

In this section, we formally define the $G$-parking and toppling critical modules. Before this, we briefly recall the corresponding ideals followed by a discussion on canonical modules.  We also discuss a characterisation of equivalent acyclic orientations and a criterion for well-definedness of candidate maps between modules. They will be turn out to be useful in the forthcoming sections.

\subsection{The $G$-Parking Function Ideal and the Toppling Ideal}\label{gparktopppre_subsect}

 We start by defining the $G$-parking function ideal of a graph.  Fix a vertex $v_2$, say of  $G$. For each non-empty subset $S$ of vertices in $V(G) \setminus \{v_2\}$, associate a monomial $m_S=\prod_{v_j \in S} x_j^{{\rm deg}_S(v_j)}$, where ${\rm deg}_S(v_j)$ is the number of edges in $G$ one of whose vertices is $v_j$ and the other vertex is in the complement $\bar{S}=V \setminus S$ of $S$. The $G$-parking function ideal $M_G$ (with respect to $v_2$) is defined as  \begin{center} $M_G=\langle m_S|~\emptyset \neq S \subseteq V(G) \setminus \{v_2\} \rangle$. \end{center}
 
 Note that $M_G$ depends on the choice of a vertex and  we take this vertex to be $v_2$.  For $M_{G/e}$ (recall that $e=(v_1,v_2)$), we take this to be the vertex $v_{1,2}$ (the vertex obtained by contracting $v_1$ and $v_2$). 

 The toppling ideal of a graph is a binomial ideal that captures the chip firing moves on $G$. It has been studied in several works recently, for instance \cite{PerPerWil11}, \cite{ManStu13}. We briefly recall its definition here.  Let $n \geq 2$.  Let $Q_G=D_G-A_G$ be the Laplacian matrix of $G$, where $A_G$ is the vertex-vertex adjacency matrix of $G$ and $D_G={\rm diag}({\rm val}(v_1),\dots,{\rm val}(v_n))$ is the diagonal matrix with its diagonal entries as the valencies of the corresponding vertices. Let $L_G$, the Laplacian lattice of $G$, be the sublattice of $\mathbb{Z}^n$ generated  by the rows (or equivalently the columns) of $Q_G$.  Since the graph $G$ is connected, the Laplacian matrix $Q_G$ has rank $n-1$. Hence, the lattice $L_G$ also has rank $n-1$ and is a finite index sublattice of the root lattice $A_{n-1}$. The toppling ideal $I_G$ of $G$ is the lattice ideal of the Laplacian lattice $L_G$. By definition, 

\begin{center} $I_G= \langle {\bf x^u}-{\bf x^v}|~{\bf u},~{\bf v} \in \mathbb{Z}_{\geq 0}^{n}, ~{\bf u}-{\bf v}  \in L_G\rangle$.\end{center}

\subsection{The Canonical Module of a Graded Ring}\label{can_subsect}

Recall from the introduction that both the critical modules are defined, up to a twist, as canonical modules of certain quotients of the polynomial ring. Hence, we start by briefly recalling the notion of canonical module of a graded ring. We refer to \cite[Chapter 3]{BruHer98} and \cite[Chapter 13, Section 4]{MilStu05} for a more detailed treatment of this topic.  However, we only deal with the critical modules in terms of their free presentations and these can be described in terms of the data of the underlying graph. Hence, a reader may choose to skip this subsection (and in principle, also the definition of the critical modules) and directly proceed to the free presentation of the critical modules presented in the next subsection.

Let $\mathfrak{R}$ be a graded ring with a unique homogenous maximal ideal  $\mathfrak{m}$ that is also maximal in the usual (ungraded) sense and of Krull dimension $\kappa$. Following \cite[Definition 3.6.8]{BruHer98} an $\mathfrak{R}$-module $\mathfrak{C}$ is called a \emph{canonical module} $\omega_\mathfrak{R}$ of $\mathfrak{R}$ if 
\begin{center}
$
{\rm Ext}_{\mathfrak{R}}^{i}(\mathfrak{R}/\mathfrak{m},\mathfrak{C}) \cong
\begin{cases}
0, \text{ for $i \neq \kappa$},\\
\mathfrak{R}/\mathfrak{m}, \text{otherwise}.
\end{cases}
$.
\end{center}
Note that the isomorphisms above are homogenous isomorphisms.  Note that the definition of a canonical module does not guarantee its existence and is, in general, a subtle issue.  
By \cite[Proposition 3.6.9]{BruHer98}, a canonical module if it exists is unique up to homogenous isomorphism.

 The existence of a canonical module for the $G$-parking and toppling ideals follows from the following facts.  By Example \cite[3.6.10]{BruHer98}, the standard graded polynomial ring $R=\mathbb{K}[x_1,\dots,x_n]$ over any field $\mathbb{K}$   is a Gorenstein ring and hence, has a canonical module: $R(-n)$ i.e., $R$ twisted \footnote{Recall that for a graded $\mathfrak{R}$-module $\mathfrak{M}$,  the $i$-th twist $\mathfrak{M}(i)$ of $\mathfrak{M}$, for an integer $i$, is the $\mathfrak{R}$-module defined as $(\mathfrak{M}(i))_j=\mathfrak{M}_{i+j}$. } by $-n$. 
 Furthermore, from \cite[Part (b), Proposition 3.6.12]{BruHer98} it follows that any Cohen-Macaulay, graded quotient ring of the graded polynomial ring has a canonical module. Since both $R/I_G$ and $R/M_G$ are graded Cohen-Macaulay rings of depth and dimension one (\cite[Proposition 7.3]{PerPerWil11} and from their minimal free resolutions in \cite{ManSchWil15}), they have canonical modules that are unique up to homogenous isomorphism. Furthermore, by  \cite[Part (b), Proposition 3.6.12]{BruHer98}, the canonical module has an explicit description as ${\rm Ext}^{t}_R(R/I,R(-n))$, where $t$ is the height of $I$.  Hence, the canonical module can be computed by applying graded ${\rm Hom}_R(\_,R(-n))$ to a minimal free resolution of $R/I$ as an $R$-module and taking the $t$-th homology of the resulting complex. In particular, the canonical module is isomorphic to the cokernel of the dual (with the appropriate twist) of the highest differential in the minimal free resolution of $R/I$.  Since the minimal free resolution of the toppling ideal and the $G$-parking function ideal can be described in purely combinatorial terms, this leads to a combinatorial description of their canonical module.  One point to note before moving to the definition of the two critical modules is that we primarily regard the critical modules as (graded) modules over the polynomial ring and not as modules  over the corresponding quotient ring.
\subsection{Critical Modules and their Free Presentations}

In this subsection, we define both the critical modules, and then discuss their free presentations.   Recall that $\ell$ is the number of loops of $G$.

%We then implement each step described in the outline of the proof of Theorem \ref{GparkTutte_theo} and Theorem \ref{Tutteses_theo}.  

%Recall that $R=\mathbb{K}[x_1,\dots,x_n]$ is the polynomial ring in $n$-variables with coefficients in $\mathbb{K}$ and 

\begin{definition} {\rm \bf ($G$-parking Critical Module)}
The $G$-parking critical module ${\rm CPark}_G$ is defined as $\omega_{R/M_G}(\ell)$, where $\omega_{R/M_G}$ is the canonical module of $R/M_G$ and $M_G$ is the $G$-parking function ideal of $G$ .
\end{definition}

\begin{definition} {\rm \bf (Toppling Critical Module)}
The toppling critical module ${\rm CTopp}_G$ is defined as $\omega_{R/I_G}(\ell)$, where $\omega_{R/I_G}$ is the canonical module of $R/I_G$ and $I_G$ is the toppling ideal of $G$.
\end{definition}

%%%Laplacian lattice, Toppling Ideals, Critical Modules.

\subsubsection{A Free Presentation of the $G$-Parking Critical Module}\label{freepresgparkcan_sect}

We recall a free presentation of ${\rm CPark}_G$ that is implicit in the minimal free resolution of $M_G$ \cite[Section 4]{ManSchWil15}. Recall that the minimal generators of ${\rm CPark}_G$ are labelled by acyclic orientations on $G$ with a unique sink at $v_2$. The minimal first syzygies of  ${\rm CPark}_G$ are labelled by acyclic orientations $\mathcal{A}$ on connected partition graphs $G/(v_i,v_j)$  (where $v_i$ and $v_j$ are connected by an edge) with a unique sink at the partition containing $v_2$. We now describe the relation corresponding to such a pair $(\mathcal{A},G/(v_i,v_j))$. Suppose that $\mathcal{A}_{(v_i,v_j)^{+}}$ and $\mathcal{A}_{(v_i,v_j)^{-}}$ are the acyclic orientations on $G$ obtained by further orienting every edge between $v_i$ and $v_j$ such that $v_i$ and $v_j$ is the source respectively.  Let $m_{i,j}$ be the number of edges between $(v_i,v_j)$.  Note that at least one of $\mathcal{A}_{(v_i,v_j)^{+}}$ and $\mathcal{A}_{(v_i,v_j)^{-}}$  has a unique sink at $v_2$.
 The relation corresponding to $(\mathcal{A},G/(v_i,v_j))$ is the following: 
\[\begin{cases}
x_i^{m_{i,j}} \cdot \mathcal{A}_{(v_i,v_j)^{+}}, \text{ if $j=2$ or $\mathcal{A}_{(v_i,v_j)^{-}}$ does not have a unique sink,}\\
x_j^{m_{i,j}} \cdot \mathcal{A}_{(v_i,v_j)^{-}},  \text{ if $i=2$ or $\mathcal{A}_{(v_i,v_j)^{+}}$ does not have a unique sink},\\
x_i^{m_{i,j}} \cdot \mathcal{A}_{(v_i,v_j)^{+}}+x_j^{m_{i,j}} \cdot \mathcal{A}_{(v_i,v_j)^{-}}, \text{ otherwise}.
\end{cases}
\]

See Example \ref{triangle_ex} for the case of a triangle. We refer to this relation as the standard syzygy corresponding to the pair $(\mathcal{A},G/(v_i,v_j))$. We refer to each of the above three types of syzygies as type one, two and three respectively.  
As mentioned in Subsection \ref{can_subsect}, the correctness of this free presentation follows from the characterisation of the canonical module of $R/M_G$ as ${\rm Ext}^{n-1}(R/M_G,R(-n))$.

\subsubsection{A Free Presentation of the Toppling Critical Module}\label{freepres_subsect}

In this subsection, we recall a free presentation of the toppling critical module ${\rm CTopp}_G$, that is also implicit in \cite[Section 3]{ManSchWil15}. Recall that the toppling critical module has a minimal generating set that is in bijection with the equivalence classes of acyclic orientations on $G$ with a unique sink at $v_2$. This equivalence class is defined by declaring two acyclic orientations $\mathcal{A}_1$ and $\mathcal{A}_2$ to be equivalent if the associated divisors $D_{\mathcal{A}_1}$ and $D_{\mathcal{A}_2}$ are linearly equivalent. By $[\mathcal{A}]$, we denote the minimal generator corresponding to the equivalence class of $\mathcal{A}$.

In the following, we describe a minimal generating set for the first syzygies of ${\rm CTopp}_G$. This minimal generating set is in bijection with equivalence classes of acyclic orientations on connected partition graphs $\mathcal{P}_{i,j}$ of $G$ of size $n-1$.  The graph $\mathcal{P}_{i,j}$ is obtained by contracting a pair of adjacent vertices $(v_i,v_j)$ of $G$ i.e., by contracting all the edges between $(v_i,v_j)$ simultaneously. Note that the equivalence class of acyclic orientations on $\mathcal{P}_{i,j}$ is defined as before by treating $\mathcal{P}_{i,j}$ as a graph. This syzygy corresponding to the equivalence class of the acyclic orientation $\mathcal{A}$ on $\mathcal{P}_{i,j}$ has the following explicit description. Suppose that $\mathcal{A}_{(v_i,v_j)^{+}}$ and $\mathcal{A}_{(v_i,v_j)^{-}}$ are acyclic orientations on $G$ obtained from $\mathcal{A}$ by further orienting all edges between $(v_i,v_j)$ such that the source is $v_i$ and $v_j$ respectively. The syzygy corresponding to $\mathcal{P}_{i,j}$ is given by \begin{center}$x_i^{m_{i,j}}[\mathcal{A}_{(v_i,v_j)^{+}}]+x_j^{m_{i,j}}[\mathcal{A}_{(v_i,v_j)^{-}}]$,\end{center} where $m_{i,j}$ is the number of edges between $v_i$ and $v_j$.  Note that we have assumed that $\mathbb{K}$ has characteristic two.  See \cite[Example 2.6]{ManSchWil15} for example of the kite graph.  We know from \cite{ManSchWil15} that this does not depend on the choice of representatives in the equivalence class of $\mathcal{A}$. The corresponding argument is essentially the same as Lemma \ref{psican_lem}.  In the following, we refer to this minimal generating set and its syzygies as the standard generating set and the standard syzygies for the critical module respectively.  

Next, we extend the notion of standard generating set and standard syzygies to ${\rm CTopp}_G \otimes_R R_e$ as an $R_e$-module.   By the right exactness of the tensor product functor, we know that a generating set for the $R_e$-module ${\rm CTopp}_G \otimes_R R_e$ and for its syzygies can be obtained from the corresponding sets for ${\rm CTopp}_G$ by tensoring each element with $1$ (the multiplicative identity of $R_e$).  We refer to these sets as the standard generating set and the standard syzygies of ${\rm CTopp}_G \otimes_R R_e$.  The standard syzygies  of ${\rm CTopp}_G \otimes_R R_e$ are obtained by replacing $x_i$ by $x_{1,2}$ from the corresponding elements in ${\rm CTopp}_G$ whenever $x_i$ is $x_1$ or $x_2$. 
The correctness of this free presentation follows from the characterisation of the canonical module of $R/I_G$ as ${\rm Ext}^{n-1}(R/I_G,R(-n))$, see Subsection \ref{can_subsect} for more details.

\subsection{Equivalence of Acyclic Orientations}\label{equiv_subsect}

Recall that we defined two acyclic orientations $\mathcal{A}_1$ and $\mathcal{A}_2$ on $G$ to be equivalent if their associated divisors $D_{\mathcal{A}_1}=\sum_{v}({\rm outdeg}_{\mathcal{A}_1}(v)-1)(v)$  and $D_{\mathcal{A}_2}=\sum_{v}({\rm outdeg}_{\mathcal{A}_2}(v)-1)(v)$ are linearly equivalent, where ${\rm outdeg}_{\mathcal{A}_i}(v)$ is the outdegree of $v$ with respect to the acyclic orientation $\mathcal{A}_i$.  The following characterisation of this equivalence in terms of reversal of a source or a sink from \cite{Mos72,Bac17} turns out be very useful. A source-sink reversal oppositely orients all the edges incident on a source or a sink. 
It is immediate that the resulting orientation is acyclic and is equivalent to the original one. The converse also holds. 

\begin{theorem}\cite{Mos72,Bac17}\label{acycequiv_theo}
Acyclic orientations $\mathcal{A}_1$ and $\mathcal{A}_2$ are equivalent if and only there is a sequence of source-sink reversals transforming $\mathcal{A}_1$ to $\mathcal{A}_2$. 
\end{theorem}

This characterisation allows us to define a metric $d$ on the set of equivalent acyclic orientations as follows: 
\begin{center}

$d(\mathcal{A}_1,\mathcal{A}_2)$ is the minimum number of source or sink reversals transforming $\mathcal{A}_1$ to $\mathcal{A}_2$. 

\end{center}
Note that $d$ satisfies the metric axioms.

\subsection{Criterion for Well-definedness of Candidate Maps between Modules}

In this subsection, we record a criterion for the well-definedness of a candidate map $f$ between two finitely presented modules $M_1$ and $M_2$ over a commutative ring $R$. This criterion is well known, we include proofs for completeness and easy access. The candidate map $f$ is given by specifying its image on a generating set of $M_1$, and the modules $M_1$ and $M_2$ are given in terms of a finite free presentation. 

\begin{proposition}\label{relpre_prop}
The candidate map $f$ is well-defined if and only if it preserves a generating set of the first syzygy module of $M_1$.
  \end{proposition}
  \begin{proof}
   The direction $\Rightarrow$ is immediate. For the converse, note $M_1 \cong R^{n_1}/S_1$ and $M_2 \cong R^{n_2}/S_2$, where $n_1,~n_2$ are the cardinalities of the corresponding generating sets and,  $S_1$ and $S_2$ are the first syzygy modules of $M_1$ and $M_2$ respectively (with respect to the chosen generating sets).  
   The map $f$ is well-defined as a map between free modules i.e., $f: R^{n_1} \rightarrow R^{n_2}$, we need to show that it descends to a map on the corresponding quotients. For this, it suffices to show that the image of $f$ on $S_1$ is contained in $S_2$. 
   Since, $f$ takes a generating set of $S_1$ to $S_2$, it takes every element of $S_1$ to an element in $S_2$.
  \end{proof}
%  for every zero relation $\sum_i r_i g_i=0$ between the generators $\{g_i\}$ of $M_1$ in the finite free presentation of $M_1$, the element $\sum_i r_i f(g_i)$ in $M_2$ is also zero. Furthermore, $f$ is well-defined if and only if
  
We  prove the well-definedness of the candidate maps $\psi_0,~\psi_1$ and $\phi_0,~\phi_1$ via Proposition \ref{relpre_prop} using the free presentation of the $G$-parking and toppling critical modules described in the previous subsections. We also use the following method to construct module maps.

\begin{proposition}  
   Suppose that $M_1$ and $M_2$ are given as co-kernels of maps between free modules $G_1 \xrightarrow{\ell_1} F_1$ and $G_2 \xrightarrow{\ell_2} F_2$ respectively, then the maps $G_1 \xrightarrow{\varrho_g} G_2$ and $F_1 \xrightarrow{\varrho_f} F_2$ between free modules specify a unique homomorphism $f: M_1 \rightarrow M_2$ if the diagram
\[ \begin{tikzcd}
G_1 \arrow{r}{\ell_1} \arrow[swap]{d}{\varrho_g} & F_1 \arrow{d}{\varrho_f} \\%
G_2\arrow{r}{\ell_2}& F_2
\end{tikzcd}
\]
commutes.
 \end{proposition}
 
 \begin{proof}
 It suffices to show that $\varrho_f$ takes every element in the image of $\ell_1$ to an element in the image of $\ell_2$.  Hence, for an element $b \in G_1$ consider $\varrho_f(\ell_1(b)) \in F_2$. Since the diagram commutes,  $\varrho_f(\ell_1(b))=\ell_2(\varrho_g(b))$. Hence, $\varrho_f(\ell_1(b))$ is in the image of $\ell_2$. 
 \end{proof}
 
%%%Criterion for proving well-definedness of maps between modules. 

\section{The $G$-Parking Tutte Short Exact Sequence}\label{gparktutteses_Sect}

In this section, we prove Theorem \ref{GparkTutte_theo} that states that the $G$-parking Tutte sequence is short exact. We starting by showing the well-definedness of the candidate maps $\psi_0$ and $\phi_0$. 

\subsection{Well-definedness of $\psi_0$}

For an acyclic orientation $\mathcal{B}$ on $G/(v_i,v_j)$ for some distinct $v_i$ and $v_j$, we denote by $\mathcal{B}_{(v_i,v_j)^{+}}$ and  $\mathcal{B}_{(v_i,v_j)^{-}}$,
the acyclic orientations on $G$ obtained by orienting every edge between $v_i$ and $v_j$ such that $v_i$ and $v_j$ is the source respectively. For an edge $\tilde{e}=(v_i,v_j)$ of $G$ and an acyclic orientation $\mathcal{B}$ on $G/\tilde{e}$, we also use the notations $\mathcal{B}_{\tilde{e}^{+}}$ and $\mathcal{B}_{\tilde{e}^{-}}$ for $\mathcal{B}_{(v_i,v_j)^{+}}$  and $\mathcal{B}_{(v_i,v_j)^{-}}$ respectively.  Recall that the map $\psi_0: {\rm CPark}_{G/e} \rightarrow {\rm CPark}_{G} \otimes_R R_e$ is defined as follows. Let $\mathcal{A}$ be the minimal generator of ${\rm CPark}_{G/e}$ corresponding to an acyclic orientation with a unique sink at $v_{1,2}$. We define $\psi_0(\mathcal{A})=x_{1,2}^{m_e-1} \mathcal{A}_{e^{+}}$, where  $\mathcal{A}_{e^{+}}$ is the minimal generator in  ${\rm CPark}_{G} \otimes_R R_e$ corresponding to the acyclic orientation on $G$ obtained by further orienting $e$ such that $v_1$ is the source. 

\begin{proposition} The map $\psi_0: {\rm CPark}_{G/e} \rightarrow {\rm CPark}_{G} \otimes_R R_e$ is well-defined.  \end{proposition}

\begin{proof}  By Proposition \ref{relpre_prop},  we verify that every standard syzygy of ${\rm CPark}_{G/e}$ is preserved by the map $\psi_0$. We label the vertices of $G/e$ by $u_2,u_3,\dots,u_n$, where $u_2:=v_{1,2}$ and for $i$ from $3,\dots,n$, the vertex $u_i$ in $G/e$ corresponds to the vertex $v_i$ in $G$.  Recall that each standard syzygy of ${\rm CPark}_{G/e}$ corresponds to a pair $(\mathcal{B},P_{i,j})$, where $\mathcal{B}$ is an acyclic orientation with a unique sink at the partition containing $v_{1,2}$ on the partition graph  $P_{i,j}$ obtained by contracting a pair of vertices $(u_i,u_j)$ of $G/e$ that are connected by an edge.  

We are led to the following cases: if neither  $u_i$ nor $u_j$ is $v_{1,2}$, then we claim that $\psi_0$ maps the standard syzygy of ${\rm CPark}_{G/e}$ corresponding to $(\mathcal{B},P_{i,j})$  to the standard syzygy  of ${\rm CPark}_{G} \otimes_R R_e$ corresponding to the pair $(\mathcal{B}_{e^{+}},G/(v_i,v_j))$. Note that $\mathcal{B}_{e^{+}}$ has a unique sink at $v_2$ and hence, $(\mathcal{B}_{e^{+}},G/(v_i,v_j))$  corresponds to a standard syzygy. 

Furthermore, note that the acyclic orientation $\mathcal{B}_{(u_i,u_j)^{+}}$ (and $\mathcal{B}_{(u_i,u_j)^{-}}$ respectively) on $G/e$ has a unique sink if and only if the acyclic orientation $(\mathcal{B}_{(v_i,v_j)^{+}})_{e^{+}}$ ($(\mathcal{B}_{(v_i,v_j)^{-}})_{e^{+}}$ respectively) on $G$ obtained by further orienting $e$ such that $v_1$ is the source also has a unique sink. Hence, the type of syzygy corresponding to $(\mathcal{B},P_{i,j})$  and  $(\mathcal{B}_{e^{+}},G/(v_i,v_j))$ among the three types described in Subsection \ref{freepresgparkcan_sect} is the same. Finally, we note that if  $\mathcal{B}_{(u_i,u_j)^{+}}$ has a unique sink (at $v_{1,2}$), then $\psi_0(\mathcal{B}_{(u_i,u_j)^{+}})=x_{1,2}^{m_e-1}\cdot(\mathcal{B}_{(v_i,v_j)^{+}})_{e^{+}}$. Similarly, if  $\mathcal{B}_{(u_i,u_j)^{-}}$ has a unique sink (at $v_{1,2}$), then $\psi_0(\mathcal{B}_{(u_i,u_j)^{-}})=x_{1,2}^{m_e-1}\cdot(\mathcal{B}_{(v_i,v_j)^{-}})_{e^{+}}$. Hence, this standard syzygy corresponding to $(\mathcal{B},\mathcal{P}_{i,j})$ is preserved by $\psi_0$. 

Consider the case where one of $u_i$ or $u_j$, $u_i$ say is $v_{1,2}$.  Suppose that only $v_2$ and not $v_1$ is adjacent to $v_j$ in $G$ then consider the standard syzygy of  ${\rm CPark}_{G} \otimes_R R_e$ corresponding to $(\mathcal{B}_{e^{+}},G/(v_2,v_j))$
and note that both these standard syzygies corresponding to $(\mathcal{B},\mathcal{P}_{i,j})$ and $(\mathcal{B}_{e^{+}},G/(v_2,v_j))$ are of the second type and that $\psi_0(\mathcal{B}_{(u_2,u_j)^{-}})=x_{1,2}^{m_e-1}\cdot(\mathcal{B}_{(v_2,v_j)^{-}})_{e^{+}}$. Hence, this standard syzygy corresponding to $(\mathcal{B},\mathcal{P}_{i,j})$ is preserved by $\psi_0$. 

Suppose that  among $v_1$ and $v_2$, precisely $v_1$, or both $v_1$ and $v_2$ are adjacent to $v_j$.  Consider the standard syzygy of  ${\rm CPark}_{G} \otimes_R R_e$ corresponding to $(\mathcal{B}_{e^{+}},G/(v_1,v_j))$. If this syzygy is of the first two types, then it must be of the second type and then the syzygy  $(\mathcal{B},P_{i,j})$ must also be of the second type (since only possibly $v_j$ among $v_1$ and $v_j$ can be a sink of $\mathcal{B}_{e^{+}}$).  The syzygies are $x_j^m \cdot \mathcal{B}_{(u_1,u_j)^{-}}$ and $x_j^m \cdot (\mathcal{B}_{(v_1,v_j)^{-}})_{e^{+}}$, where $m$ is the multiplicity of the edge $(v_1,v_j)$. Note that $\psi_0(\mathcal{B}_{(u_1,u_j)^{-}})=x_{1,2}^{m_e-1}\cdot(\mathcal{B}_{(v_1,v_j)^{-}})_{e^{+}}$ and hence, this is preserved.  

Otherwise,  this syzygy is of the third type and it is of the form  $x_{1,2}^{m} \cdot \mathcal{K} +x_j^{m}  \cdot (\mathcal{B}_{e^{+}})_{(v_1,v_j)^{-}}$, where $\mathcal{K}=(\mathcal{B}_{e^{+}})_{(v_1,v_j)^{+}}$ is the acyclic orientation on $G$ obtained from $ (\mathcal{B}_{e^{+}})_{(v_1,v_j)^{-}}$ by reversing the orientation of every edge between the vertices $(v_1,v_j)$. We consider the syzygy corresponding to the acyclic orientation induced by  $\mathcal{K}$ on $G/(v_1,v_2)$ (note that the edge $e$ is contractible on $\mathcal{K}$).  Since $v_2$ is a sink this syzygy is of the form $x_{1,2}^{m_e} \cdot \mathcal{K}$, where $m_e$ is the multiplicity of the edge $e$.  Hence, if $m \geq m_e$ we obtain the syzygy $x_j^{m} \cdot  (\mathcal{B}_{e^{+}})_{(v_1,v_j)^{-}}$ from the standard syzygies as $(x_{1,2}^{m} \cdot \mathcal{K}+x_j^m \cdot  (\mathcal{B}_{e^{+}})_{(v_1,v_j)^{-}})-x_{1,2}^{m-m_e}(x_{1,2}^{m_e} \cdot \mathcal{K})$. Otherwise, we obtain the syzygy $x_{1,2}^{m_e-m}(x_j^m \cdot  (\mathcal{B}_{e^{+}})_{(v_1,v_j)^{-}})$ as $x_{1,2}^{m_e-m}(x_{1,2}^{m} \cdot \mathcal{K}+x_j^m \cdot  (\mathcal{B}_{e^{+}})_{(v_1,v_j)^{-}})-(x_{1,2}^{m_e} \cdot \mathcal{K})$.  Hence, $x_{1,2}^{m_e-1}x_j^m \cdot  (\mathcal{B}_{e^{+}})_{(v_1,v_j)^{-}}$ is a syzygy of ${\rm CPark}_G \otimes_R R_e$. Finally, note that $x_{1,2}^{m_e-1}x_j^m \cdot  (\mathcal{B}_{e^{+}})_{(v_1,v_j)^{-}}=\psi_0(x_j^m \cdot  \mathcal{B}_{(v_{1,2},u_j)^{-}})$ and that $x_j^m \cdot  \mathcal{B}_{(v_{1,2},u_j)^{-}}$ is the standard syzygy corresponding to the pair $(\mathcal{B},P_{i,j})$ to complete the proof. 
 \end{proof}

 \subsection{Well-definedness of $\phi_0$}

We need the following combinatorial lemma for the well-definedness of $\phi_0$. Recall that an edge $e$ of $G$ is said to be contractible on an acyclic orientation  $\mathcal{A'}$ on $G$ if the orientation $\mathcal{A'}/e$ induced by $\mathcal{A'}$ on $G/e$ is acyclic.
 
 \begin{lemma}\label{notcontra_lem}
Suppose that $\mathcal{A}$ is an acyclic orientation on $G$ with a unique sink at $v_2$. Suppose that there is a directed edge from $v_1$ to $v_j \neq v_2$, then the edge $e=(v_1,v_2)$ is not contractible. 
\end{lemma}
\begin{proof}
In order to show that $e$ is not contractible, we need to exhibit a directed path from $v_1$ to $v_2$ that is not equal to the edge $e$. Construct a directed walk starting from $v_j$ by picking arbitrary outgoing edges. This walk cannot repeat vertices since  $\mathcal{A}$ is acyclic and  hence, it will terminate since $G$ is a finite graph. Furthermore, it terminates in $v_2$ since it is the unique sink. Appending the directed edge $(v_1,v_j)$ to the beginning of this walk yields the required directed path from $v_1$ to $v_2$.  
\end{proof}

Recall that we defined the map $\phi_0: {\rm CPark}_{G} \otimes_R R_e \rightarrow {\rm CPark}_{G \setminus e} \otimes_R R_e$ as the following. If $m_e=1$, then

 %We denote by $\mathcal{A''}$ the unique acyclic orientation on $G \setminus e$ with a unique sink at $v_2$ that is equivalent to the acyclic orientation $\mathcal{A'} \setminus e$ on $G \setminus e$. {\color{blue} is this necessary?} 
 
\[
\phi_0(\mathcal{A'})=
\begin{cases}
\mathcal{A'}\setminus e,\text{ if the edge $e$ is not contractible on $\mathcal{A'}$},\\
0, \text{ otherwise.}
\end{cases}
\]
If $m_e>1$, then $\phi_0(\mathcal{A'})=\mathcal{A'} \setminus e$ for every standard generator $\mathcal{A}'$ of ${\rm CPark}_{G} \otimes_R R_e$.
 
 \begin{proposition} The map $\phi_0$ is well-defined. \end{proposition}

\begin{proof} By Proposition \ref{relpre_prop},  we verify that every standard syzygy of ${\rm CPark}_{G} \otimes_R R_e$ is preserved by the map $\phi_0$. We start by noting that any standard syzygy corresponding to the partition graph $\mathcal{P}_{1,2}$ of $G$ obtained by contracting $(v_1,v_2)$ is of the form $x_{1,2}^{m_e} \cdot \mathcal{A}_{e^{+}}$, where $\mathcal{A}$ is an acyclic orientation on $G/e$. If $m_e=1$, then $\phi_0$ maps it to zero.  Since $e$ is contractible on $\mathcal{A}_{e^{+}}$ and hence, $\phi_0(\mathcal{A}_{e^{+}})=0$. If $m_e>1$, then  $\phi_0(x_{1,2}^{m_e} \cdot \mathcal{A}_{e^{+}})=x_{1,2}^{m_e} \cdot (\mathcal{A}_{e^{+}} \setminus e)=0$ since $x_{1,2}^{m_e-1} \cdot (\mathcal{A}_{e^{+}} \setminus e)$ is the standard syzygy of ${\rm CPark}_{G \setminus e} \otimes_R R_e$ corresponding to the pair $(\mathcal{A},G/(v_1,v_2))$.

We now consider standard syzygies corresponding to other partition graphs.  If one of the vertices is $v_2$ and the other vertex is $v_j \neq v_1$, then the standard syzygy corresponding to $(\mathcal{A},\mathcal{P}_{2,j})$ is of type two  and is of the form $x_j^m \cdot \mathcal{A}_{(v_2,v_j)^{-}}$. The  map $\phi_0$ takes it to the standard syzygy corresponding to $(\mathcal{A} \setminus e,\mathcal{P}_{2,j} \setminus e)$.

If one of the vertices is $v_1$ and the other vertex is $v_j \neq v_2$, then if the standard syzygy $\mathcal{S}$ is of type three and is of form $x_1^{m} \cdot \mathcal{A}_{(v_1,v_j)^{+}}+x_j^m \cdot \mathcal{A}_{(v_1,v_j)^{-}}$. We consider two cases:

{\bf Case I:}  Suppose that $m_e =1$.  Note that by Lemma \ref{notcontra_lem}, the edge $e$ is not contractible on $\mathcal{A}_{(v_1,v_j)^{+}}$.  We have the following two cases: If $e$ is not contractible on  $\mathcal{A}_{(v_1,v_j)^{-}}$, then we have $\phi_0(\mathcal{A}_{(v_1,v_j)^{-}})=\mathcal{A}_{(v_1,v_j)^{-}}\setminus e$ and $\phi_0(\mathcal{A}_{(v_1,v_j)^{+}})=\mathcal{A}_{(v_1,v_j)^{+}}\setminus e$.  Furthermore, the standard syzygy corresponding to the acyclic orientation  $(\mathcal{A}/(v_1,v_j)) \setminus e$ on $(G \setminus e)/(v_1,v_j)$ (this is the acyclic orientation induced by $\mathcal{A}$ on the graph obtained by contracting $(v_1,v_j)$ and deleting $e$) is a type three syzygy since in both $\mathcal{A}_{(v_1,v_j)^{-}} \setminus e$ and $\mathcal{A}_{(v_1,v_j)^{+}} \setminus e$, the vertex $v_1$ has at least one outgoing edge and is hence not a sink.  This implies that both acyclic orientations have a unique sink at $v_2$.  This syzygy $\mathcal{S}'$ is of the form: $x_1^{m} \cdot \mathcal{A}_{(v_1,v_j)^{+}} \setminus e+x_j^m \cdot \mathcal{A}_{(v_1,v_j)^{-}} \setminus e$. Hence, $\phi_0$ takes $\mathcal{S}$ to $\mathcal{S}'$.  If $e$ is contractible on $\mathcal{A}_{(v_1,v_j)^{-}}$, then $\phi_0(\mathcal{A}_{(v_1,v_j)^{-}})=0$ and $\phi_0(\mathcal{A}_{(v_1,v_j)^{+}})=\mathcal{A}_{(v_1,v_j)^{+}}\setminus e$. Furthermore, the standard syzygy $\mathcal{S}'$ corresponding to the acyclic orientation $(\mathcal{A}/(v_1,v_j)) \setminus e$ is of type one and of the form $x_1^m \cdot  \mathcal{A}_{(v_1,v_j)^{+}} \setminus e$. Hence, $\phi_0$ takes $\mathcal{S}$ to $\mathcal{S}'$. 

 %This maps to the standard syzygy $x_1^{m} \cdot  (\mathcal{A}_{(v_1,v_j)^{-}}\setminus e)$ of ${\rm GC}_{G \setminus e} \otimes_R R_e$ corresponding to the acyclic orientation $(\mathcal{A}/(v_1,v_j)) \setminus e$ on $(G \setminus e)/(v_1,v_j)$ (this is the acyclic orientation induced by $\mathcal{A}$ on the graph obtained by contracting $(v_1,v_j)$ and deleting $e$). This is because $e$ is contractible on $(\mathcal{A}_{(v_1,v_j)^{-}},G)$ since $v_1$ is the sink and by Lemma \ref{notcontra_lem}, this is because the edge $e$ is not contractible on $(\mathcal{A}_{(v_1,v_j)^{+}},G)$. 

{\bf Case II:}   If $m_e>1$, then $\phi_0(\mathcal{A}_{(v_1,v_j)^{-}})=\mathcal{A}_{(v_1,v_j)^{-}}\setminus e$ and $\phi_0(\mathcal{A}_{(v_1,v_j)^{+}})=\mathcal{A}_{(v_1,v_j)^{+}}\setminus e$. Also, this standard syzygy maps to the standard syzygy of ${\rm CPark}_{G \setminus e} \otimes_R R_e$ corresponding to $(\mathcal{A},\mathcal{P}_{i,j})$. Suppose that this standard syzygy of ${\rm CPark}_{G} \otimes_R R_e$ is of type two, then it is of the form $x_1^{m} \cdot \mathcal{A}_{(v_1,v_j)^{-}}$ and $\phi_0$ (irrespective of $m_e$) maps it to the standard syzygy $x_1^{m} \cdot \mathcal{A}_{(v_1,v_j)^{-}} \setminus e$ corresponding to the acyclic orientation $(\mathcal{A}/(v_1,v_j)) \setminus e$ on $G \setminus e$, note that this is also a syzygy of type two.

 If none of the vertices is $v_1$ or $v_2$, then the standard syzygy corresponding to $(\mathcal{A},\mathcal{P}_{i,j})$ is mapped to the standard syzygy corresponding to $\mathcal{A} \setminus e$ on $\mathcal{P}_{i,j} \setminus e$ independent of $m_e$.  Note that the type of these two standard syzygies are the same and $\phi_0$ maps $\mathcal{A}_{(v_i,v_j)^{+}}$ to $(\mathcal{A}_{(v_i,v_j)^{+}}) \setminus e$ and maps $\mathcal{A}_{(v_i,v_j)^{-}}$ to $(\mathcal{A}_{(v_i,v_j)^{-}}) \setminus e$ (when they are well-defined). \end{proof}

\subsection{Complex Property}

\begin{proposition} The $G$-parking Tutte sequence in Theorem \ref{GparkTutte_theo} is a complex of graded $R_e$-modules. \end{proposition}
\begin{proof}
We show that the property of a complex is satisfied at each homological degree. At homological degrees zero and two, this is immediate. At homological degree one, we need to show that $\phi_0(\psi_0(b))=0$ for every $b \in {\rm CPark}_{G/e}$. It suffices to prove this for every standard generator $\mathcal{A}$ of  ${\rm CPark}_{G/e}$. To see this,  consider the case where $m_e=1$, we have $\psi_0(\mathcal{A})=\mathcal{A}_{e^{+}}$ and $\phi_0(\mathcal{A}_{e^{+}})=0$ since the edge $e$ is contractible on the acyclic orientation $\mathcal{A}_{e^{+}}$ on $G$. If $m_e>1$, then $\psi_0(\mathcal{A})=x_{1,2}^{m_e-1} \cdot \mathcal{A}_{e^{+}}$ and $\phi_0(\mathcal{A}_{e^{+}})=x_{1,2}^{m_e-1} \cdot (\mathcal{A}_{e^{+}} \setminus e)=0$ since  $x_{1,2}^{m_e-1} \cdot (\mathcal{A}_{e^{+}} \setminus e)$ is the standard syzygy (of ${\rm CPark}_{G \setminus e} \otimes_R R_e$) corresponding to the acyclic orientation $\mathcal{A}$ on $G/(v_1,v_2)$. 
\end{proof}

\subsection{The Kernel of $\psi_0$}

\begin{proposition}\label{kerpsi_prop} The kernel of $\psi_0: {\rm CPark}_{G/e} \rightarrow {\rm CPark}_{G} \otimes_R R_e$ is equal to $x_{1,2} \cdot {\rm CPark}_{G/e}$. \end{proposition}

\begin{proof}
The inclusion $x_{1,2} \cdot {\rm CPark}_{G/e}$ in the kernel of $\psi_0$ is immediate since $\psi_0(x_{1,2} \cdot \mathcal{A})=x_{1,2}^{m_e} \cdot {\mathcal{A}}_{e^{+}}=0$ since $x_{1,2}^{m_e}  \cdot {\mathcal{A}}_{e^{+}}$ is the standard syzygy corresponding to the acyclic orientation $\mathcal{A}$ of the partition graph $G/(v_1,v_2)$.  

For the other direction, consider an element $\alpha=\sum_{\mathcal{A}} p_{\mathcal{A}} \cdot \mathcal{A}$ in the kernel of $\psi_0$.  We show that  the coefficients $p_{\mathcal{A}}$ can be chosen such that $x_{1,2}|p_{\mathcal{A}}$ for each $\mathcal{A}$.
Since $\alpha \in {\rm ker}(\psi_0)$, we obtain $x_{1,2}^{m_e-1}\sum_{\mathcal{A}}p_{\mathcal{A}} \cdot \mathcal{A}_{e^{+}}=0$. Furthermore, since the map $\mathcal{A} \rightarrow \mathcal{A}_{e^{+}}$ regarded as a map between sets of acyclic orientations (with a unique sink at a fixed vertex) is injective, we note that $x_{1,2}^{m_e-1}\sum_{\mathcal{A}}p_{\mathcal{A}} \cdot \mathcal{A}_{e^{+}}$ is a syzygy of ${\rm CPark}_G \otimes_{R} R_e$. 

Hence, it can be written as an $R_e$-linear combination of the standard syzygyies of ${\rm CPark}_G$. Hence, 

\begin{equation}\label{syz_eqn} x_{1,2}^{m_e-1}\sum_{\mathcal{A}}p_{\mathcal{A}} \cdot \mathcal{A}_{e^{+}}=\sum_{(\mathcal{B},\mathcal{P}_{i,j})}r_{(\mathcal{B},\mathcal{P}_{i,j})}s_{(\mathcal{B},\mathcal{P}_{i,j})} \end{equation}
 , where $s_{(\mathcal{B},\mathcal{P}_{i,j})}$ is the standard syzygy corresponding to the acyclic orientation $\mathcal{B}$ (with a unique sink at the partition containing $v_2$) on the partition graph $\mathcal{P}_{i,j}$ and $r_{(\mathcal{B},\mathcal{P}_{i,j})} \in R_e$. Note that Equation (\ref{syz_eqn}) is an equation in the free $R_e$-module of rank equal to the number of acyclic orientations with a unique sink at $v_2$.
 
Consider the case $m_e=1$.  Next, we observe that if $(i,j)$ (as an unordered pair) is not $(1,2)$, then the syzygy $s_{(\mathcal{B},\mathcal{P}_{i,j})}$  is the image of a standard syzygy of ${\rm CPark}_{G/e}$, namely the standard syzygy corresponding to the acyclic orientation $\mathcal{B}/e$ on $(G/(v_i,v_j))/e$ obtained by contracting $e$.  These syzygies can be cleared out by regarding $\alpha$ as $\alpha-\sum_{(i,j) \neq (1,2)}r_{(\mathcal{B},\mathcal{P}_{i,j})}s_{\mathcal{B}/e,G/(v_i,v_j))/e}$ and using the expansion in Equation (\ref{syz_eqn}) for $\psi_0(\alpha)$.  Hence, we can assume that the standard syzygies in Equation (\ref{syz_eqn}) all correspond to $\mathcal{P}_{1,2}$. A standard syzygy corresponding to $\mathcal{P}_{1,2}$ is of the form $x_{1,2} \cdot \mathcal{A}_{e^{+}}$ for some acyclic orientation $\mathcal{A}$ on $G/(v_1,v_2)$. This implies that each coefficient $p_{\mathcal{A}}$ divides $x_{1,2}$. This completes the proof for $m_e=1$.

More generally, if $m_e \geq 1$, then we multiply both sides of the Equation (\ref{syz_eqn}) by $x_{1,2}^{m_e-1}$.  The argument then proceeds similar to the case $m_e=1$. If $(i,j)$ (as an unordered pair) is not $(1,2)$, then  $x_{1,2}^{m_e-1} \cdot s_{(\mathcal{B},\mathcal{P}_{i,j})}$  is the image of a standard syzygy of ${\rm CPark}_{G/e}$, namely the standard syzygy corresponding to the acyclic orientation $\mathcal{B}/e$ on $(G/(v_i,v_j))/e$ obtained by contracting $e$.  Hence, these syzygies can be cleared out and we can assume that only terms corresponding to $\mathcal{P}_{1,2}$ appear. A standard syzygy corresponding to $\mathcal{P}_{1,2}$ is of the form $x_{1,2}^{m_e} \cdot \mathcal{A}_{e^{+}}$ for some acyclic orientation $\mathcal{A}$ on $G/(v_1,v_2)$. Hence, we conclude that $x_{1,2}^{2m_e-1}|(x_{1,2}^{2m_e-2} \cdot p_{\mathcal{A}})$ and hence, $x_{1,2}|p_{\mathcal{A}}$. We conclude that $\alpha$ is contained in $x_{1,2} \cdot {\rm CPark}_{G/e}$.
\end{proof}

\subsection{Exactness}

\begin{proposition} The $G$-parking Tutte complex is a short exact sequence. \end{proposition}
\begin{proof}
We show the exactness of the $G$-parking Tutte complex at every homological degree. At homological degree zero, the exactness follows from Proposition \ref{kerpsi_prop}. At homological degree two, the exactness is equivalent to the surjectivity of $\phi_0$. 
To see the surjectivity of $\phi_0$, we consider two cases. Suppose that $m_e=1$. Note that for every standard generator $\mathcal{A}''$ of ${\rm CPark}_{G \setminus e} \otimes_R R_e$, there is the acyclic orientation $\mathcal{A}'$ on $G$ obtained by further orienting $e$ such that $v_1$ is the source. The edge $e$ is not contractible on $\mathcal{A}'$ since there is at least one edge other than $e$ with a source at $v_1$ and we can now apply Lemma \ref{notcontra_lem}.  Hence, $\phi_0$ takes $\mathcal{A}'$ to $\mathcal{A}''$. Since every standard generator of ${\rm CPark}_{G \setminus e} \otimes_R R_e$ is in the image of an element in $\phi_0$ we conclude that $\phi_0$ is surjective. If $m_e>1$, then this is immediate from the construction of $\phi_0$ since every acyclic orientation on $G \setminus e$ with a unique sink at $v_2$ gives rise to an acyclic orientation on $G$ with a unique sink at $v_2$ by further orienting $e$ such that $v_1$ is the source.  

We turn to homological degree one. We must show that the kernel of $\phi_0$ is equal to the image of $\psi_0$ which in turn is $\langle x_{1,2}^{m_e-1} \mathcal{A}_{e^{+}} \rangle$, where $\mathcal{A}$ ranges over acyclic orientations on $G/e$ with a unique sink at $v_{1,2}$.  Suppose that $b=\sum_{\mathcal{A'}}p_{\mathcal{A'}} \mathcal{A'}$ is an element in the kernel of $\phi_0$. We know that $\sum_{\mathcal{A'}}p_{\mathcal{A'}} \phi_0(\mathcal{A'})=0$.  Suppose that $m_e>1$. Note that the map $\mathcal{A'} \rightarrow \mathcal{A'} \setminus e$ at the level of sets is a bijection.  Hence,  $\sum_{\mathcal{A'}}p_{\mathcal{A'}} \cdot \phi_0(\mathcal{A'})$ is a syzygy of  ${\rm CPark}_{G \setminus e} \otimes_R R_e$ and  can be written as an $R_e$-linear combination of the standard syzygies of ${\rm CPark}_{G \setminus e} \otimes_R R_e$. More precisely,  we have:

\begin{equation}\label{syzeqnses_lem}
\sum_{\mathcal{A'}}p_{\mathcal{A'}} \cdot \phi_0(\mathcal{A'})=\sum_{(\mathcal{A''},\mathcal{P}_{i,j})}r_{(\mathcal{A''},\mathcal{P}_{i,j})}\cdot s_{(\mathcal{A''},\mathcal{P}_{i,j})}
\end{equation}
, where $r_{(\mathcal{A''},\mathcal{P}_{i,j})} \in R_e$ and $s_{(\mathcal{A''},\mathcal{P}_{i,j})}$ is the standard syzygy of ${\rm CPark}_{G \setminus e} \otimes_R R_e$ corresponding to the acyclic orientation $\mathcal{A''}$ on the partition graph $\mathcal{P}_{i,j}$ of $G \setminus e$. Note that this equation is on the free $R_e$-module of rank equal to the number of acyclic orientations on $G \setminus e$ with a unique sink at $v_2$.
 
 Suppose that $(i,j) \neq (1,2)$ as an unordered pair. By the construction of the map $\phi_0$, the standard syzygy is the image of $\phi_0$ over the standard syzygy of ${\rm CPark}_{G} \otimes_R R_e$ corresponding to the same pair $(\mathcal{A''},\mathcal{P}_{i,j})$. Hence, these syzygies can be cleared exactly as in the proof of Proposition \ref{kerpsi_prop} and we can assume that $(i,j)=(1,2)$ in the right hand side of Equation (\ref{syzeqnses_lem}). Since every standard syzygy corresponding to  $(\mathcal{A''},\mathcal{P}_{1,2})$ is of the form $x_{1,2}^{m_e-1} \cdot \mathcal{A''}_{e^{+}}$, we conclude that $x_{1,2}^{m_e-1}$ divides $p_{\mathcal{A''}}$ for every $\mathcal{A''}$ and hence, $b \in \langle x_{1,2}^{m_e-1} \mathcal{A}_{e^{+}} \rangle$.  Thus, the Tutte complex is exact in homological degree one for $m_e>1$.
 
 We turn to the case $m_e=1$.  Consider any element of the form $\sum_{\mathcal{A'}}p_{\mathcal{A'}}\cdot \mathcal{A'}$, where $\mathcal{A'}$ is an acyclic orientation on $G$ such that $e$ is not contractible on it. It suffices to show that if such a $\sum_{\mathcal{A'}}p_{\mathcal{A'}}\cdot \mathcal{A'}$  is in kernel of $\phi_0$, then it is also in the image of $\psi_0$. By Equation (\ref{syzeqnses_lem}),  we assume by a clearing argument  that $(i,j)=(1,j)$ where $j \neq 2$ in Equation  (\ref{syzeqnses_lem}). Furthermore, we can assume that the standard syzygies that appear in Equation  (\ref{syzeqnses_lem}) corresponding to $\mathcal{P}_{1,j}$ are of the form $x_{1,2}^{m_{1,j}} \cdot \mathcal{A''}$.   These are precisely standard syzygies that change type from type three to type one from $G$ to $G \setminus e$. This means that the acyclic orientation $\mathcal{A''}$ is such that there is a unique vertex $v_j \notin \{v_1,v_2\}$ that is adjacent to $v_1$, and with $v_1$ as the source of every edge between $v_1$ and $v_j$. Hence, we conclude that $x_{1,2}^{m_{1,j}}$ divides $p_{\mathcal{A'}}$.   
 
 Next, note that   $x_{1,2}^{m_{1,j}} \cdot  \mathcal{A'}+x_j^{m_{1,j}} \cdot \mathcal{B'}$  is a standard syzygy of ${\rm CPark}_{G} \otimes_R R_e$ corresponding to the acyclic orientation $\mathcal{A'}/(1,j)=\mathcal{B'}/(1,j)$ on $G/(1,j)$ and $ \mathcal{B'}$ is the acyclic orientation on $G$ obtained from $\mathcal{A'}$ by reversing the orientation of every edge between $(1,j)$. Note that the edge $e$ is contractible on  $\mathcal{B'}$ (since $\phi_0(\mathcal{B'})=0$) and hence, it is of the form $\mathcal{A}_{e^{+}}$ for an acyclic orientation $\mathcal{A}$ on $G/e$. Hence, $x_{1,2}^{m_{1,j}} \cdot  \mathcal{A'}=-x_j^{m_{1,j}} \cdot \mathcal{B'}$ is in the image of $\psi_0$. This completes the proof of exactness at homological degree one.   
  \end{proof}

\section{The Toppling Tutte Short Exact Sequence}\label{topptutteses_sect}

In this section, we detail the proof of Theorem \ref{Tutteses_theo}.

\subsection{Well-definedness of $\psi_1$}

We start by recalling the construction of the candidate map $\psi_1$. Suppose that $e$ is an edge of multiplicity $m_e$ between the vertices $v_1$ and $v_2$. Let $\mathcal{A}$ be an acyclic orientation on $G/e$, and let $\mathcal{A}_{e^{+}}$ and $\mathcal{A}_{e^{-}}$ be orientations on $G$ obtained by further orienting $e$ such that its source is $v_1$ and $v_2$ respectively. Note that since $\mathcal{A}$ is acyclic, the orientations $\mathcal{A}_{e^{+}}$ and $\mathcal{A}_{e^{-}}$ are also acyclic.  The candidate map $\psi_1$ takes $[\mathcal{A}]$ to $x_{1,2}^{m_e-1}[A_{e^{+}}]+x_{1,2}^{m_e-1}[A_{e^{-}}]$ in ${\rm CTopp}_G \otimes_R R_e$.  We first show that this association is independent of the choice of representatives in the equivalence class of $\mathcal{A}$.

%{\color{blue} Make the correction: only the sum is well-defined.}

\begin{lemma}\label{psican_lem} Suppose that $\mathcal{A}$ is an acyclic orientation on $G/e$. The equivalence classes of the acyclic orientations $\mathcal{A}_{e^{+}}$ and $\mathcal{A}_{e^{-}}$ on $G$ are independent of the choice of representatives in the equivalence class of $\mathcal{A}$. \end{lemma}
\begin{proof}
%{\color{blue} Explain the idea of proof in more detail in the preliminaries: for instance what does sink reversal mean?}
Suppose that acyclic orientations $\mathcal{A}_1$ and $\mathcal{A}_2$ on $G/e$ are equivalent.  We know from \cite[Section 3.1]{BakNor07} that for any vertex $u$ of $G/e$, there exists a (unique) acyclic orientation $\mathcal{A}_{\rm uni}$ with a unique sink at $u$ that is equivalent to $\mathcal{A}_1$ and $\mathcal{A}_2$. Furthermore, by Theorem \ref{acycequiv_theo},  there is a sequence of source-sink reversals that transform them to $\mathcal{A}_{\rm uni}$.

Take $u=v_{1,2}$ and note that a sink reversal at $v_{1,2}$ can be avoided in this sequence.  This allows us to perform precisely the same sequence of source-sink reversals for $({\mathcal{A}_1})_{e^{+}}$ and $({\mathcal{A}_2})_{e^{+}}$. If $v_{1,2}$ is a sink for an acyclic orientation $\mathcal{A}$ on $G/e$, then $v_2$ is a sink for $\mathcal{A}_{e^{+}}$ and $v_1$ is a sink for $\mathcal{A}_{e^{-}}$. Hence, for $i=1,2$ these operations transform $(\mathcal{A}_i)_{e^{+}}$  and $(\mathcal{A}_i)_{e^{-}}$ into an acyclic orientation on $G$ with a unique sink at $v_2$ and $v_1$ respectively. From \cite{Mos72,Bac17}, this implies that the acyclic orientations $({\mathcal{A}_1})_{e^{+}}$  and $({\mathcal{A}_2})_{e^{+}}$ on $G$ are equivalent and that $({\mathcal{A}_1})_{e^{-}}$  and $({\mathcal{A}_2})_{e^{-}}$ are also equivalent. 
\end{proof}

Next, we show this candidate map induces a map between the  toppling critical modules ${\rm CTopp}_{G/e}$ and ${\rm CTopp}_G \otimes_R R_e$.  We show this using Proposition \ref{relpre_prop}. 

\begin{lemma}\label{psi1_lem}
The candidate map $\psi_1$ is well-defined. 
\end{lemma}

\begin{proof}
%{\color{blue} Improve the notation?}
We use Proposition \ref{relpre_prop} to show that $\psi_1$ is well-defined. In other words, we show that $\psi_1$ preserves the standard syzygies of ${\rm CTopp}_{G/e}$.  Using the combinatorial description of the syzygies of ${\rm CTopp}_{G/e}$ in Subsection \ref{freepres_subsect}, we know that the generators of the first syzygy module of ${\rm CTopp}_{G/e}$ are in one to one correspondence with equivalence classes of acyclic orientations on contractions of pairs of adjacent vertices of $G/e$. For a pair of adjacent vertices $(v_i,v_j)$ and an acyclic orientation $\mathcal{A}$ on the contraction of $(v_i,v_j)$ in $G/e$, the corresponding syzygy of ${\rm CTopp}_{G/e}$ is given by $x_i^{m} [\mathcal{A}_{(v_i,v_j)^{+}}]-x_j^m [\mathcal{A}_{(v_i,v_j)^{-}}]$, where $m$ is the number of edges between the pair $(v_i,v_j)$ and $\mathcal{A}_{(v_i,v_j)^{+}}$ and $\mathcal{A}_{(v_i,v_j)^{-}}$ are acyclic orientations obtained from $\mathcal{A}$ by further orienting all the edges between $(v_i,v_j)$ so that the source is $v_i$ and $v_j$ respectively. We must show that $x_i^{m} \psi_1([\mathcal{A}_{(v_i,v_j)^{+}}])-x_j^m \psi_1([\mathcal{A}_{(v_i,v_j)^{-}}])$ is a syzygy of ${\rm CTopp}_G \otimes_R R_e$. Note that 

\begin{center}

$x_i^{m} \psi_1([\mathcal{A}_{(v_i,v_j)^{+}}])-x_j^m \psi_1([\mathcal{A}_{(v_i,v_j)^{-}}])=x_i^m([(\mathcal{A}_{(v_i,v_j)^{+}})_{e^{+}}]+[(\mathcal{A}_{(v_i,v_j)^{+}})_{e^{-}}])-x_j^m([(\mathcal{A}_{(v_i,v_j)^{-}})_{e^{+}}]+[(\mathcal{A}_{(v_i,v_j)^{-}})_{e^{-}}])$.
 
 \end{center}
Suppose that none of $v_i$ and $v_j$ is $v_{1,2}$, then $x_i^m[(\mathcal{A}_{(v_i,v_j)^{+}})_{e^{+}}]-x_j^m[(\mathcal{A}_{(v_i,v_j)^{-}})_{e^{+}}]$ and $x_i^m[(\mathcal{A}_{(v_i,v_j)^{+}})_{e^{-}}]-x_j^m[(\mathcal{A}_{(v_i,v_j)^{-}})_{e^{-}}]$ are standard syzygies of ${\rm CTopp}_{G} \otimes_R R_e$. Similarly, if $v_i=v_{1,2}$ and exactly one of $v_1$ and $v_2$ is adjacent to $v_j$ in $G$, then $x_i^m[(\mathcal{A}_{(v_i,v_j)^{+}})_{e^{+}}]-x_j^m[(\mathcal{A}_{(v_i,v_j)^{-}})_{e^{+}}]$ and $x_i^m[(\mathcal{A}_{(v_i,v_j)^{+}})_{e^{-}}]-x_j^m[(\mathcal{A}_{(v_i,v_j)^{-}})_{e^{-}}]$ are standard syzygies of ${\rm CTopp}_{G} \otimes_R R_e$. 

Finally, consider the case where $v_i=v_{1,2}$, and both $v_1$ and $v_2$ are adjacent to $v_j$ in $G$. In other words, if there is a triangle between $v_1$, $v_2$ and $v_j$ in $G$, then an analogous argument does not hold. We employ a different argument.  
We express $\mathcal{S}=x_{1,2}^m[(\mathcal{A}_{(v_{1,2},v_j)^{+}})_{e^{+}}]-x_j^m[(\mathcal{A}_{(v_{1,2},v_j)^{-}})_{e^{+}}]$  as an $R_e$-linear combination of two standard syzygies of ${\rm CTopp}_G \otimes_R R_e$. These standard syzygies $\mathcal{S}_1$ and $\mathcal{S}_2$ are the following: the syzygy $\mathcal{S}_1$  corresponds to the acyclic orientation $\mathcal{A}_1$ on the contraction of the pair $(v_1,v_j)$ in $G$ defined as follows: $\mathcal{A}_1$ agrees with $\mathcal{A}_{(v_{1,2},v_j)^{+}}$ on all the common edges and the edge $v_1,v_2$ is further oriented such that $v_1$ is the source. The other one $\mathcal{S}_2$ corresponds to the acyclic orientation $\mathcal{A}_2$ on the contraction of the pair $(v_2,v_j)$ in $G$ where  $\mathcal{A}_2$ agrees with  $\mathcal{A}_{(v_{1,2},v_j)^{-}}$ on all the common edges and the edge $v_1,v_2$ is further orientated such that $v_1$ is the source, see Figure \ref{prooffig1}. Note that the former syzygy $\mathcal{S}_1$ is $x_{j}^{m_1}[(\mathcal{A}_{(v_{1,2},v_j)^{-}})_{e^{+}}]-x_{1,2}^{m_1}[\mathcal{K}]$ and the later syzygy $\mathcal{S}_2$ is  $x_{j}^{m_2}[\mathcal{L}]-x_{1,2}^{m_2}[(\mathcal{A}_{(v_{1,2},v_j)^{+}})_{e^{+}}]$, where $m_1$ and $m_2$ is the number of edges in $G$ between $(v_1,v_j)$ and $(v_2,v_j)$ respectively.  Furthermore, observe that $\mathcal{K}=\mathcal{L}$ (we do not make  $\mathcal{K}$ and $\mathcal{L}$ more explicit since we do not invoke it) and that $m=m_1+m_2$. Hence,  $\mathcal{S}=x_j^{m_2}\mathcal{S}_1+x_{1,2}^{m_1}\mathcal{S}_2$. 
Similarly, we express $x_{1,2}^{m}[(\mathcal{A}_{(v_{1,2},v_j)^{+}})_{e^{-}}]-x_j^m[(\mathcal{A}_{(v_{1,2},v_j)^{-}})_{e^{-}}]$ as an $R_e$-linear combination of standard syzygies of ${\rm CTopp}_G$ by interchanging $v_1$ and $v_2$ in the above construction. This completes the proof of the well-definedess of $\psi_1$.
\end{proof}

\begin{figure}
  \centering
  \includegraphics[width=6cm]{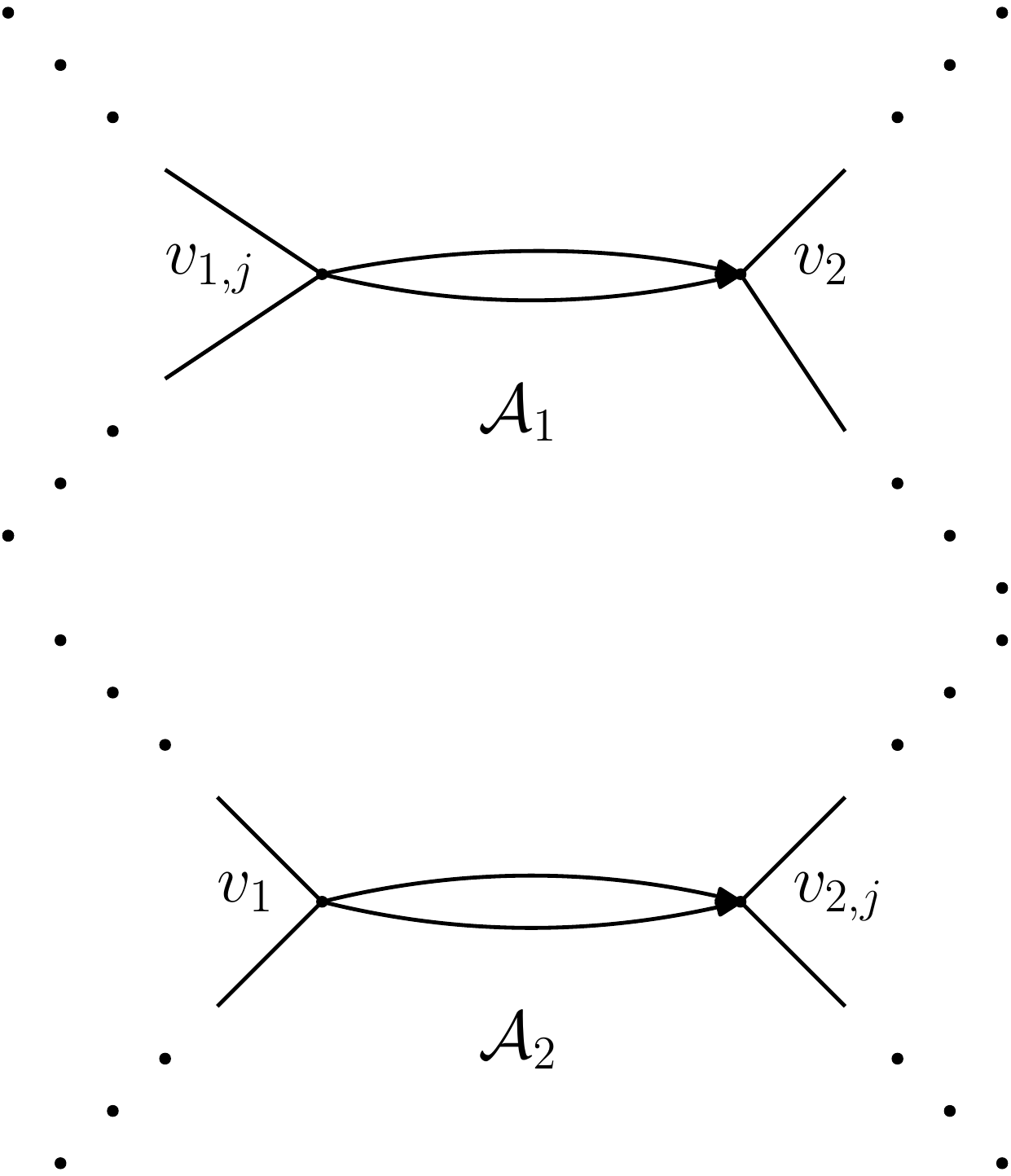}
\caption{The acyclic orientations $\mathcal{A}_1,\mathcal{A}_2$ in the proof of Lemma \ref{psi1_lem}}\label{prooffig1}
\end{figure}

\subsection{Well-definedness of $\phi_1$}

 Recall that the map $\phi_1: {\rm CTopp}_G \otimes_R R_e \rightarrow {\rm CTopp}_{G \setminus e} \otimes_R R_e$ takes the generator $[\mathcal{A'}]$ of ${\rm CTopp}_G$, corresponding to an acyclic orientation $\mathcal{A}$ on $G$,  to $[\mathcal{A'} \setminus e]$ in ${\rm CTopp}_{G \setminus e}$. The following lemma shows that the association $[\mathcal{A}] \rightarrow [\mathcal{A} \setminus e]$ in the candidate map $\phi_1$ does not depend on the choice of representatives.  

\begin{lemma} Let $\mathcal{A}_1$ and $\mathcal{A}_2$ be equivalent acyclic orientations on $G$. The acyclic  orientations $\mathcal{A}_1 \setminus e$ and $\mathcal{A}_2 \setminus e$ on $G \setminus e$ are equivalent.  \end{lemma}
\begin{proof}
Using the characterisation of equivalent acyclic orientations (Theorem \ref{acycequiv_theo}), we know that there is a source-sink reversal sequence transforming $\mathcal{A}_1$ to $\mathcal{A}_2$. Since, any source or sink in an acyclic orientation $\mathcal{A}$ on $G$ remains so in the acyclic orientation $\mathcal{A} \setminus e$ on $G \setminus e$, we can perform the same source-sink reversal sequence to transform $\mathcal{A}_1 \setminus e$ to  $\mathcal{A}_2 \setminus e$. Hence, 
$\mathcal{A}_1 \setminus e$ and $\mathcal{A}_2 \setminus e$ are equivalent. 
\end{proof}

\begin{lemma} The candidate map $\phi_1$ is well-defined.  \end{lemma}

\begin{proof} Using Proposition \ref{relpre_prop}, it suffices to prove that $\phi_1$ preserves the standard syzygies of ${\rm CTopp}_{G} \otimes_R R_e$.   The proofs lends itself into two cases: the first case corresponds to the standard syzygy arising from contracting $G$ by a pair of vertices $(v_i,v_j)$ other than $(v_1,v_2)$ and the second case corresponds to the standard syzygy arising from contracting $G$ by $(v_1,v_2)$.

Consider the standard syzygy corresponding to an acyclic orientation $\mathcal{A}$ on $G/(v_i,v_j)$ i.e., $G$ contracted by the pair of vertices $(v_i,v_j) \neq(v_1,v_2)$ that are connected by an edge. This syzygy is  $x_i^m[\mathcal{A}_{(v_i,v_j)^{+}}]-x_j^m[\mathcal{A}_{(v_i,v_j)^{-}}]$, where $m$ is the number of edges between $v_i$ and $v_j$. By construction, $\phi_1([\mathcal{A}_{(v_i,v_j)^{+}}])=[\mathcal{A}_{(v_i,v_j)^{+}} \setminus e]$ and  $\phi_1([\mathcal{A}_{(v_i,v_j)^{-}}])=[\mathcal{A}_{(v_i,v_j)^{-}} \setminus e]$. Note that  $x_i^m\phi_1([\mathcal{A}_{(v_i,v_j)^{+}}])-x_j^m\phi_1([\mathcal{A}_{(v_i,v_j)^{-}}])$ is a standard syzygy of ${\rm CTopp}_{G \setminus e} \otimes_R  R_e$ and corresponds to the acyclic orientation $\mathcal{A} \setminus e$ on $(G \setminus e)/(v_i,v_j)$ i.e., $G \setminus e$ contracted by the pair of vertices $(v_i,v_j)$.  Hence, this standard syzygy on ${\rm CTopp}_{G} \otimes_R R_e$ is preserved. 

Suppose that the standard syzygy corresponds to an acyclic orientation $\mathcal{A}$ on the $G/(v_1,v_2)$. This standard syzygy is  $x_{1,2}^{m_e}[\mathcal{A}_{(v_1,v_2)^{+}}]-x_{1,2}^{m_e}[\mathcal{A}_{(v_1,v_2)^{-}}]$, where $m_e$ is the number of edges between the pair $(v_1,v_2)$.  Suppose that $m_e=1$. In this case, there is exactly one edge between $(v_1,v_2)$ and since $\mathcal{A}_{(v_1,v_2)^{+}}\setminus e=\mathcal{A}_{(v_1,v_2)^{-}}\setminus e$ we have  $\phi_1([\mathcal{A}_{(v_1,v_2)^{+}}])=\phi_1([\mathcal{A}_{(v_1,v_2)^{-}}])$. Hence, 
$x_{1,2}\phi_1([\mathcal{A}_{(v_1,v_2)^{+}}])-x_{1,2}\phi_1([\mathcal{A}_{(v_1,v_2)^{-}}])=0$.  

Suppose that $m_e>1$ i.e., there are multiple edges between $(v_1,v_2)$. In this case, $\phi_1([\mathcal{A}_{(v_i,v_j)^{+}}])=[\mathcal{A}_{(v_i,v_j)^{+}} \setminus e]$ and  $\phi_1([\mathcal{A}_{(v_i,v_j)^{-}}])=[\mathcal{A}_{(v_i,v_j)^{-}} \setminus e]$. Furthermore, we know that 
 $x_{1,2}^{m_e-1}\phi_1([\mathcal{A}_{(v_i,v_j)^{+}}])-x_{1,2}^{m_e-1}\phi_1([\mathcal{A}_{(v_i,v_j)^{-}}])$ is a standard syzygy of ${\rm CTopp}_{G \setminus e} \otimes_R  R_e$. This corresponds to the acyclic orientation induced by $\mathcal{A}$ on $(G \setminus e)/(v_1,v_2)=G/(v_1,v_2)$. 
 Multiplying by $x_{1,2}$ throughout, we conclude that $x_{1,2}^{m_e}\phi_1([\mathcal{A}_{(v_i,v_j)^{+}}])-x_{1,2}^{m_e}\phi_1([\mathcal{A}_{(v_i,v_j)^{-}}])=0$. This completes the proof. 
 \end{proof}

\subsection{Complex Property}

In this subsection, we show that the toppling Tutte sequence  in Theorem \ref{Tutteses_theo} is a complex of $R_e$-modules. At homological degree zero and two,  this property is immediate.  Only homological degree one requires an argument and we address it in the following proposition.

\begin{proposition}
The kernel of the map $\phi_1$ contains the image of the map $\psi_1$. In other words, for any element $b \in {\rm CTopp}_{G/e}/{\rm ker}(\psi_1)$ we have $\phi_1(\psi_1(b))=0$.
\end{proposition}
\begin{proof}
Note that, since  $\psi_1$ and $\phi_1$ are $R_e$-module maps, it suffices to prove the statement for the projection of the standard generating set of ${\rm CTopp}_{G/e}$ on ${\rm CTopp}_{G/e}/{\rm ker}(\psi_1)$. Consider an element $[\mathcal{A}]$ of the standard generating set of ${\rm CTopp}_{G/e}$. We use the same notation for its projection in ${\rm CTopp}_{G/e}/{\rm ker}(\psi_1)$ and consider $\psi_1([\mathcal{A}])$. By definition, $\psi_1([\mathcal{A}])=x_{1,2}^{m_e-1}[\mathcal{A}_{e^{+}}]+x_{1,2}^{m_e-1}[\mathcal{A}_{e^{-}}] \in {\rm CTopp}_G \otimes_R R_e$.  Hence, $\phi_1(\psi_1([\mathcal{A}]))=x_{1,2}^{m_e-1}\phi_1([\mathcal{A}_{e^{+}}])+x_{1,2}^{m_e-1}\phi_1([\mathcal{A}_{e^{-}}])$.

We have two cases: $m_e=1$ i.e., there is precisely one edge $e$ between $(v_1,v_2)$. Hence, $\phi_1(\psi_1([\mathcal{A}]))=\phi_1([\mathcal{A}_{e^{+}}])+\phi_1([\mathcal{A}_{e^{-}}])$.
 In this case,  $\phi_1([\mathcal{A}_{e^{+}}])=\phi_1([\mathcal{A}_{e^{-}}])$ since $\mathcal{A}_{e^{+}} \setminus e=\mathcal{A}_{e^{-}} \setminus e$ and hence,  $\phi_1(\psi_1([\mathcal{A}]))=0$ (note that $\mathbb{K}$ has characteristic two).
 Suppose that $m_e>1$ i.e., there are multiple edges between $(v_1,v_2)$. In this case, $\phi_1([\mathcal{A}_{e^{+}}])=[\mathcal{A}_{e^{+}} \setminus e]$ and $\phi_1([\mathcal{A}_{e^{-}}])=[\mathcal{A}_{e^{-}} \setminus e]$. Note that
  $x_{1,2}^{m_e-1}[\mathcal{A}_{e^{+}} \setminus e] +x_{1,2}^{m_e-1}[\mathcal{A}_{e^{-}} \setminus e]$ is a standard syzygy of ${\rm CTopp}_{G \setminus e} \otimes_R R_e$: the standard syzygy corresponding to the acyclic orientation induced by $\mathcal{A}$ on the contraction of $G \setminus e$ by $(v_1,v_2)$, which in turn is $G/(v_1,v_2)$. Hence, $\phi_1(\psi_1([\mathcal{A}]))=x_{1,2}^{m_e-1}\phi_1([\mathcal{A}_{e^{+}}])+x_{1,2}^{m_e-1}\phi_1([\mathcal{A}_{e^{-}}])=0$.  \end{proof}

\subsection{Exactness}
We show that the toppling Tutte complex in Theorem \ref{Tutteses_theo}  is exact in every homological degree.  Since, by construction, the map $\psi_1: {\rm CTopp}_{G/e}/{\rm ker}(\psi_1) \rightarrow {\rm CTopp}_{G} \otimes_R R_e$ is injective by construction, the Tutte complex  is exact in homological degree zero. We are left with showing the exactness in homological degrees one and two. They are handled in the following two propositions.

\begin{proposition}
The map $\phi_1$ is surjective. Hence, the toppling Tutte complex is exact in homological degree two. 
\end{proposition}
\begin{proof}
It suffices to prove that every element in the standard generating set of ${\rm CTopp}_{G \setminus e} \otimes_R R_e$ is in the image, under the map $\phi_1$, of some element in  ${\rm CTopp}_{G} \otimes_R R_e$.  To see this, note that any equivalence class of acyclic orientations on $G \setminus e$ has an acyclic orientation $\mathcal{A'}\setminus e$ with a unique sink at $v_2$. This acyclic orientation $\mathcal{A'}\setminus e$  can be extended to an acyclic orientation $\mathcal{A'}$ on $G$ by further orienting $e$ so that the source is $v_1$. By construction, $\phi_1([\mathcal{A'}])=[\mathcal{A'} \setminus e]$.  Hence, $\phi_1$ is surjective.
\end{proof}

\begin{proposition}\label{toppzeroh1_prop}
The kernel of $\phi_1$ is equal to the image of $\psi_1$. In other words, the toppling Tutte complex is exact in homological degree one. 
\end{proposition}

\begin{proof}
By definition, the image of $\psi_1$ is equal to the submodule generated by $x_{1,2}^{m_e-1}[\mathcal{A}_{{e}^{+}}]+x_{1,2}^{m_e-1}[\mathcal{A}_{{e}^{-}}]$ over all the standard generators $[\mathcal{A}]$ of ${\rm CTopp}_{G/e}/{\rm ker}(\psi_1)$ (the projection of the standard generating set of ${\rm CTopp}_{G/e}$ onto ${\rm CTopp}_{G/e}/{\rm ker}(\psi_1)$). 
 We show that this is also the kernel of $\phi_1$. 
 
Consider an element $b=\sum_{[\mathcal{A}]}p_{[\mathcal{A}]}[\mathcal{A}]  \in {\rm ker}(\phi_1)$. Since, $\phi_1(b)=0$ it gives rise to a syzygy  in ${\rm CTopp}_{G \setminus e} \otimes_R R_e$ (possibly the trivial syzygy where the coefficient of each standard generator is zero). Hence, it can be written as an $R_e$-linear combination of the standard syzygies of ${\rm CTopp}_{G \setminus e}\otimes_R R_e$. Next, we compare the standard syzygies of ${\rm CTopp}_{G} \otimes_R R_e$ that are in the image of the standard syzygies of ${\rm CTopp}_{G \setminus e} \otimes_R R_e$. By clearing out these standard syzygies, we assume that  the syzygy corresponding to $\phi_1(b)$ is generated by the standard syzygies of ${\rm CTopp}_{G \setminus e} \otimes_R R_e$ that are not in the image of the standard syzygies of ${\rm CTopp}_G \otimes_R R_e$.  We refer to these as the relevant standard syzygies of  ${\rm CTopp}_{G \setminus e} \otimes_R R_e$. Furthermore, $b$ is generated by elements whose image with respect to $\phi_1$ is a relevant standard syzygy of ${\rm CTopp}_{G \setminus e} \otimes_R R_e$ and by sums (recall that $\mathbb{K}$ is characteristic two) of pairs of elements in a fiber over $\phi_1$ of any standard generator of ${\rm CTopp}_{G \setminus e} \otimes_R R_e$ (note that $\phi_1$ takes standard generators of ${\rm CTopp}_G \otimes_R R_e$ to standard generators of ${\rm CTopp}_{G \setminus e} \otimes_R R_e$).  To see this, consider $\phi_1(b)=\sum_{[\mathcal{A}]}p_{[\mathcal{A}]}\phi_1([\mathcal{A}])$ and collect coefficients of each standard generator $\phi_1([\mathcal{A}])$ of ${\rm CTopp}_{G \setminus e} \otimes_R R_e$. 
The sum of the terms corresponding to those  $\phi_1([\mathcal{A}])$ whose coefficient is non-zero is a syzygy of ${\rm CTopp}_{G \setminus e} \otimes_R R_e$. A simple calculation shows that for each $\phi_1([\mathcal{A}])$ whose coefficient is zero, the sum $\sum_{[\mathcal{A'}]}p_{[\mathcal{A'}]}\phi_1([\mathcal{A'}])$ of all $[\mathcal{A'}]$ that are mapped to $\phi_1([\mathcal{A}])$ by $\phi_1$ is generated by sums of pairs of elements in a fiber over $\phi_1$ of the standard generator $\phi_1([\mathcal{A}])$ of ${\rm CTopp}_{G \setminus e} \otimes_R R_e$. 
 
 We have two cases, if the multiplicity $m_e$ of the pair $(v_1,v_2)$ is precisely one i.e., there is precisely one edge between $(v_1,v_2)$. In this case, there are  no relevant standard syzygies of ${\rm CTopp}_{G \setminus e} \otimes_R R_e$. Consider two elements $[\mathcal{A}_1]$ and $[\mathcal{A}_2]$ in the fiber over $\phi_1$  of some standard generator $[\mathcal{A''}]$ of ${\rm CTopp}_{G \setminus e} \otimes_R R_e$. We show that  $[\mathcal{A}_1]+[\mathcal{A}_2]$ is the image of $\psi_1$. To show this, we consider the fiber over the map $\phi_1$ of a standard generator $[\mathcal{A''}]$ of ${\rm CTopp}_{G \setminus e} \otimes_ R R_e$.  This can be described as follows: consider all acyclic orientations $\mathcal{B''}$ on $G \setminus e$ that are equivalent to $\mathcal{A''}$ (see Subsection \ref{equiv_subsect}) and take the union of $[\mathcal{B''}_{e^{+}}]$ and $[\mathcal{B''}_{e^{-}}]$  for each orientation of the form  $\mathcal{B''}_{e^{+}}$ and $\mathcal{B''}_{e^{-}}$ on $G$ obtained from $\mathcal{B''}$ by further orienting $e$ such that $v_1$ and $v_2$ is the source respectively that is  acyclic.
  
 Consider the sum $[\mathcal{B}_1]+[\mathcal{B}_2]$ of any two elements in the fiber over $\phi_1$ of $[\mathcal{A''}]$. This means that the acyclic orientations $\mathcal{B}_1 \setminus e$ and $\mathcal{B}_2 \setminus e$ on $G \setminus e$ are equivalent.  By \cite{Mos72,Bac17}, there is a  source-sink reversal sequence transforming  $\mathcal{B}_1\setminus e$ to $\mathcal{B}_2 \setminus e$. Taking this into account, we perform an induction on the distance $d$ between the acyclic orientations $\mathcal{B}_1'':=\mathcal{B}_1 \setminus e$ and $\mathcal{B}_2'':=\mathcal{B}_2 \setminus e$ (recall the notion of distance from the Preliminaries, end of Subsection \ref{equiv_subsect}).  The base case corresponds to the distance between $\mathcal{B}_1''$ and $\mathcal{B}_2''$ being zero i.e., $\mathcal{B}_1''=\mathcal{B}_2''$.  In this case, $[(\mathcal{B}_1''])_{e^{+}}]+[(\mathcal{B}_1'')_{e^{-}}]$  is in the image of $\psi_1$ on the standard generator corresponding to the acyclic orientation on $G/e$ induced by $\mathcal{B}_1''$.  Note that since both $(\mathcal{B}_1'')_{e^{+}}$ and $(\mathcal{B}_1'')_{e^{-}}$ are acyclic orientations on $G$, the orientation $(\mathcal{B}_1'')_{e^{+}}/e=(\mathcal{B}_1'')_{e^{-}}/e$ is an acyclic orientation on $G/e$.
 
The induction hypothesis is that elements of the form $[(\mathcal{B''}_1)_{e^{\pm}}]+[(\mathcal{B''}_2)_{e^{\pm}}]$ in ${\rm CTopp}_G \otimes_R R_e$, where $\mathcal{B}_1''$ and $\mathcal{B}_2''$ are acyclic orientations on $G \setminus e$ that are equivalent and at a distance at most $d$ is in the image of $\psi_1$  for some non-negative integer $d$. For the induction step, consider equivalent acyclic orientations  $\mathcal{B}_1''$ and $\mathcal{B}_2''$ that are at a distance $d+1$. There exists an acyclic orientation $\mathcal{B}_3''$ that is equivalent to both and $d(\mathcal{B}_1'',\mathcal{B}_3'')=d$ and $d(\mathcal{B}_3'',\mathcal{B}_2'')=1$. 
 By the induction hypothesis, the sum of any pair $[(\mathcal{B''}_1)_{e^{\pm}}]+[(\mathcal{B''}_3)_{e^{\pm}}]$ is in the image of $\psi_1$.  We are left with showing that $[(\mathcal{B''}_2)_{e^{\pm}}]+[(\mathcal{B''}_3)_{e^{\pm}}]$ is in the image of $\psi_1$.
 Since  $d(\mathcal{B}_2'',\mathcal{B}_3'')=1$, there is precisely one source-sink reversal that transforms $\mathcal{B}_3''$ to $\mathcal{B}_2''$. If the source or sink that is reversed in neither $v_1$ nor $v_2$, then this vertex will remain so in $({\mathcal{B}_3''})_{e^{+}}$ and $({\mathcal{B}_3''})_{e^{-}}$. This can be reversed to obtain $({\mathcal{B}_2''})_{e^{+}}$ and $({\mathcal{B}_2''})_{e^{-}}$ respectively. Hence,  $[({\mathcal{B}_3''})_{e^{+}}]=[({\mathcal{B}_2''})_{e^{+}}]$ and $[({\mathcal{B}_3''})_{e^{-}}]=[({\mathcal{B}_2''})_{e^{-}}]$. We conclude that any element of the form $[(\mathcal{B''}_2)_{e^{\pm}}]+[(\mathcal{B''}_3)_{e^{\pm}}]$ is in the image of $\psi_1$.
 
 Suppose that a sink is reversed and this is either $v_1$ or $v_2$,  $v_1$ say. We note that $(\mathcal{B}_3'')_{e^{-}}$ is equivalent to $(\mathcal{B}_2'')_{e^{+}}$ since $v_1$ will remain a sink in  $(\mathcal{B}_3'')_{e^{-}}$ and can be reversed to obtain  $(\mathcal{B}_2'')_{e^{+}}$.  Hence, $[(\mathcal{B}_2'')_{e^{+}}]=[(\mathcal{B}_3'')_{e^{-}}]$. Hence, we know that $[(\mathcal{B}_2'')_{e^{+}}]+[(\mathcal{B}_3'')_{e^{-}}]=0$, $[(\mathcal{B}_2'')_{e^{-}}]+[(\mathcal{B}_3'')_{e^{-}}]$ and $[(\mathcal{B}_2'')_{e^{+}}]+[(\mathcal{B}_3'')_{e^{+}}]$ are in the image of $\psi_1$. We conclude that the sum   $[(\mathcal{B}_2'')_{e^{-}}]+[(\mathcal{B}_3'')_{e^{+}}]=([(\mathcal{B}_2'')_{e^{-}}]+[(\mathcal{B}_2'')_{e^{+}}])+([(\mathcal{B}_3'')_{e^{-}}]+[(\mathcal{B}_3'')_{e^{+}}])$ is in the image of $\psi_1$.  %{\color{blue} include Figure}

Similarly,  if a source is reversed and this is either $v_1$ or $v_2$, $v_1$ say, then  $(\mathcal{B''}_3)_{e^{+}}$ is equivalent to $(\mathcal{B''}_2)_{e^{-}}$. An analogous argument shows that the any pair $[(\mathcal{B''}_2)_{e^{\pm}}]+[(\mathcal{B''}_3)_{e^{\pm}}]$ is in the image of $\psi_1$. This completes the argument for the case $m_e=1$. We refer to Example \ref{fourcycletutte_ex} for the case of a four cycle.

Consider the case where $m_e>1$. The relevant standard syzygies of ${\rm CTopp}_{G \setminus e} \otimes_R R_e$ bijectively correspond to acyclic orientations on $G/(v_1,v_2)$.  The map $\phi_1$ induces a bijection between the standard generators of ${\rm CTopp}_G \otimes_R R_e$ and ${\rm CTopp}_{G \setminus e} \otimes_R R_e$.  Hence, each fiber over $\phi_1$ of the standard generators of ${\rm CTopp}_{G \setminus e} \otimes_R R_e$ has precisely one element.  Hence, the kernel of $\phi_1$ is generated by elements whose image is a relevant standard syzygy of ${\rm CTopp}_{G \setminus e} \otimes_R R_e$ and are of the form  $x_{1,2}^{m_e-1}[\mathcal{A}_{{e}^{+}}]+x_{1,2}^{m_e-1}[\mathcal{A}_{{e}^{-}}]$ over all acyclic orientations on $G/e$. Hence, the kernel of $\phi_1$ is equal to the image of $\psi_1$.  \end{proof}

\begin{example}\label{fourcycletutte_ex}
\rm
\begin{figure}
  \centering
  \includegraphics[width=10cm]{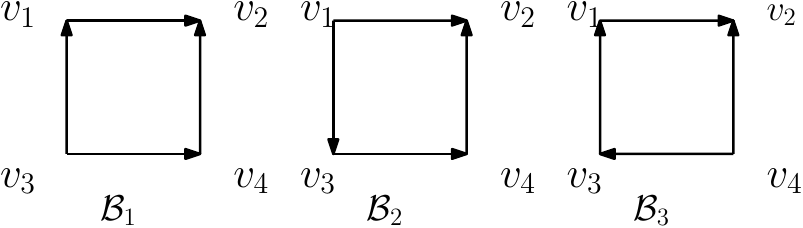}
\caption{Acyclic Orientations corresponding to the Minimal generators of the Toppling Critical Module of a Four Cycle}\label{ex2_fig}
\end{figure}

Let $G$ be the four cycle. It has three acyclic orientations $\mathcal{B}_1,\mathcal{B}_2,~\mathcal{B}_3$ with a unique sink at $v_2$ shown in Figure \ref{ex2_fig}.  Let $e=(v_1,v_2)$. The graph $G/e$ is the three cycle and has two acyclic orientations $\mathcal{A}_1:=\mathcal{B}_1/e,~\mathcal{A}_2:=\mathcal{B}_2/e$ with a unique sink at $v_2$  and $G \setminus e$ is a tree with  one acyclic orientation $\mathcal{C}$ with a unique sink at $v_2$.  The map $\psi_1$ is as follows:  
\begin{center}
$\psi_1([\mathcal{A}_1])=[\mathcal{B}_1]+[\mathcal{B}_2]$,\\
 $\psi_1([\mathcal{A}_2])=[\mathcal{B}_1]+[\mathcal{B}_3]$. \\
\end{center}
The map $\phi_1$ is as follows:

\begin{center}
$\phi_1([\mathcal{B}_1])=\phi_1([\mathcal{B}_2])=\phi_1([\mathcal{B}_3])=[\mathcal{C}]$.   
\end{center}
The element $[\mathcal{B}_2]+[\mathcal{B}_3]$ is in the kernel of $\phi_1$. It is however not an image of $\psi_1$ on the standard generators of ${\rm CTopp}_{G/e} \otimes_R R_e$ but is the image of $[\mathcal{A}_1]+[\mathcal{A}_2]$. The core of the proof of Proposition \ref{toppzeroh1_prop} is to generalise this.  \qed
\end{example}

%{\color{blue}
%\subsection{Split Exactness of the Toppling Exact Sequence}
%By splitting lemma, it suffices to construct a section $\pi_1: C_{G \setminus e} \otimes_R R_e \rightarrow C_{G} \otimes_R R_e$ of the map $\phi_1$. Define $\pi_1$ as follows: 
%\begin{center}
%$\pi_1([\mathcal{A}''])=[\mathcal{A}''_{e^{+}}]$
%\end{center}
%where $\mathcal{A}'$ 
%}

%\subsection{The Kernel of $\psi_1$}

%We now complete the proof of Theorem \ref{Tutteses_theo} by showing that the kernel of the map $\psi_1: C_{G/e} \rightarrow C_{G} \otimes_R R_e$ is $x_{1,2} \cdot C_{G/e}$.

%\begin{proposition}
%The kernel of the map $\psi_1: C_{G/e} \rightarrow C_{G} \otimes_R R_e$ is equal to $x_{1,2} \cdot C_{G/e}$.
%\end{proposition}

%%%Definition Bayer-Sturmfels Modules (include the twist).

%%%%Determine their Hilbert series. 

%%%%Use this to prove that these module are isomorphic to toppling critical modules. 

\section{Applications}

\subsection{Merino's Theorem}

We obtain Merino's theorem as a corollary to Theorem \ref{GparkTutte_theo}. The main remaining step is to show that $x_1-x_2$ is a non-zero divisor on ${\rm CPark}_G$ and ${\rm CPark}_{G \setminus e}$, and that $x_{1,2}$ is a non-zero divisor on ${\rm CPark}_{G/e}$. This is handled by the following propositions.

\begin{proposition}\label{nonzero1-prop} 
The element $x_1-x_2$ is a non-zero divisor on ${\rm CPark}_G$ and ${\rm CPark}_{G \setminus e}$. 
\end{proposition}
\begin{proof}
First, we note that  it suffices to show that  $x_1-x_2$ is a non-zero divisor of $R/M_G$ and $R/M_{G \setminus e}$ (\cite[Sections 3.3 and 3.6 ]{BruHer98}). To see this note that $x_2$ is a non-zero divisor of $R/M_G$ and $R/M_{G \setminus e}$ since the vertex $v_2$ is the sink and the ideals $M_G$ and $M_{G \setminus e}$ are generated by monomials each of which is not divisible by $x_2$.  We conclude the proof by noting that if $x_1-x_2$ is a zero divisor on $R/M_G$ (or $R/M_{G \setminus e}$), then both $x_1$ and $x_2$ are zero divisors (since $M_G$ and $M_{G \setminus e}$ are monomial ideals). This yields the required contradiction.
                                                                                                                                                                                                                                                                                                                                                                                                                                                                                                                                                                                                                                                                                                                                                                                                                                                                                                                                                                   \end{proof}

\begin{proposition}\label{nonzero2-prop}
The element $x_{1,2}$ is a non-zero divisor on ${\rm CPark}_{G/e}$.
\end{proposition}
\begin{proof}
From \cite[Proposition 3.3.3]{BruHer98}, it suffices to show that  $x_{1,2}$ is a non-zero divisor of $R/M_{G/e}$. This follows from the fact that $v_{1,2}$ is the sink for $M_{G/e}$ and hence, $M_{G/e}$ is generated by monomials each of which is not divisible by $x_{1,2}$.
\end{proof}

We are now ready to deduce Merino's theorem as a corollary. 

\begin{corollary} {\rm({\bf Merino's Theorem})}
The $K$-polynomial of ${\rm CPark}_G$  is the Tutte evaluation $T_G(1,t)$, where $T_G(x,y)$ is the Tutte polynomial of $G$. 
\end{corollary} 

\begin{proof}

We verify the base cases first. It consists of trees with $n$ vertices and $\ell$ loops in total and graph on two vertices with $m$ multiple edges and $\ell$ loops in total (recall that Theorem \ref{GparkTutte_theo} requires $G$ to have at least three vertices). In the first case, the $K$-polynomial of  ${\rm CPark}_G$ is $t^{\ell}$. On the other hand, the Tutte polynomial $T_G(x,y)$ is $x^{n-1}y^{\ell}$. In this case, we verify that $T_G(1,t)=t^{\ell}$.  In the second case, the $K$-polynomial $K_{{\rm CPark}_G}(t)$ of  ${\rm CPark}_G$ is $t^{\ell+m-1}+t^{\ell+m-2}+\cdots+t^{\ell}$ (use the fact that ${\rm CPark}_G$ is Gorenstein). The Tutte polynomial $T_G(x,y)$ is $y^{\ell+m-1}+y^{\ell+m-2}+\dots+y^{\ell +1}+x \cdot y^{\ell}$ and hence, $T_G(1,t)=K_{{\rm CPark}_G}(t)$. 

Since the Hilbert series is additive under short exact sequence of graded modules, we obtain the following equation from Theorem \ref{GparkTutte_theo}:
\begin{center}
${\rm Hil}_{{\rm CPark}_G \otimes_R R_e}(t)={\rm Hil}_{{\rm CPark}_{G/e}/(x_{1,2} \cdot {\rm CPark}_{G/e})}(t)+{\rm Hil}_{{\rm CPark}_{G \setminus e} \otimes_R R_e}(t)$.
\end{center}
Using the fact that the $G$-parking critical module has Krull dimension one and Propositions \ref{nonzero1-prop}, \ref{nonzero2-prop}, we conclude that 
\begin{center}
${\rm Hil}_{{\rm CPark}_G \otimes_R R_e}(t)=K_{{\rm CPark}_G}(t),{\rm Hil}_{{\rm CPark}_{G/e}/(x_{1,2} \cdot {\rm CPark}_{G/e})}(t)=K_{{\rm CPark}_{G/e}}(t)$ and ${\rm Hil}_{{\rm CPark}_{G \setminus e} \otimes_R R_e}(t)=K_{{\rm CPark}_{G \setminus e}}(t)$.
\end{center}
 Hence we obtain:

\begin{center}
$K_{{\rm CPark}_G}(t)=K_{{\rm CPark}_{G/e}}(t)+K_{{\rm CPark}_{G \setminus e}}(t)$.
\end{center}
This completes the proof of Merino's theorem. \end{proof}

%This yields the Tutte delete-contraction formula for the $K$-polynomial of the $G$-parking critical module and the universality of the Tutte polynomial \cite{MonMer11} yields the corollary. {\color{blue} need to check two more conditions: when $e$ is bridge and $e$ is loop?}

%This yields the Tutte delete-contraction formula for the $K$-polynomial of the $G$-parking critical module and the universality of the Tutte polynomial \cite{MonMer11} yields the corollary. {\color{blue} need to check two more conditions: when $e$ is bridge and $e$ is loop?}

\begin{remark}\label{hileq_rem}
\rm Note that the Hilbert series and hence, the $K$-polynomials of ${\rm CPark}_G$ and ${\rm CTopp}_G$ are equal.  To see this, note that $R/I_G$ and $R/M_G$ have the same Hilbert function \cite{ManStu13} and hence, their canonical modules share the Hilbert function \cite[Corollary 4.3.8]{BruHer98}. 
\end{remark}

%We now use this to prove the deletion-contraction formula for alternating numbers:

\subsection{Deletion-Contraction Formula for Alternating Numbers}

We first note that $\beta_{i,j}(H)$ as defined in the introduction is the $(i,j)$-th graded Betti number of the both the $G$-parking and the toppling critical module.

\begin{proposition}\label{critbetti_prop}
The number $\beta_{i,j}(H)$ is the $(i,j)$-th graded Betti number of the $G$-parking critical module ${\rm CPark}_H$ and the toppling critical module ${\rm CTopp}_H$.
\end{proposition}
\begin{proof}
This is an immediate consequence of the description of the Betti numbers of $R/M_H$ and $R/I_H$ from \cite{ManSchWil15}, and from the graded version of \cite[Corollary 3.3.9]{BruHer98} i.e., the relation between the Betti numbers of a graded Cohen-Macaulay module and its canonical module. 
\end{proof}

{\bf Proof of Proposition \ref{delete-contra_prop}:}
First, note that by expressing the Hilbert series of a graded module in terms of its Betti numbers, we obtain 
\begin{center}
$(\sum_{i,j}(-1)^i\beta_{i,j}(G)t^j)/(1-t)^n=K_G(t)/(1-t)$,\\
$(\sum_{i,j}(-1)^i\beta_{i,j}(G \setminus e)t^j)/(1-t)^n=K_{G \setminus e}(t)/(1-t)$ and 
$(\sum_{i,j}(-1)^i\beta_{i,j}(G/e)t^j)/(1-t)^{n-1}=K_{G/e}(t)/(1-t)$.
\end{center}

Applying Merino's theorem and comparing the coefficients of powers of $t$ yields the deletion-contraction formula for the alternating numbers. \qed

\subsection{The Tutte Long Exact Sequence of Tor}

{\bf Proof of Theorem \ref{vanimpeq_theo}:} 
Consider the long exact sequence in Tor associated with the short exact sequence $0 \rightarrow {\rm CPark}_{G/e}(-1) \xrightarrow{\cdot x_{1,2}} {\rm CPark}_{G/e} \rightarrow {\rm CPark}_{G/e}/(x_{1,2} \cdot {\rm CPark}_{G/e}) \rightarrow 0$ and restrict to the $j$-th degree. Note that $\beta_{r,s-1}(G/e)$ is the $(r,s)$-th graded Betti number of ${\rm CPark}_{G/e}(-1)$. Hence, if $\beta_{i,j}(G/e)=\beta_{i-1,j-1}(G/e)=0$, then the $(i,j)$-th Betti number of ${\rm CPark}_{G/e}/(x_{1,2} \cdot {\rm CPark}_{G/e})$ is zero. Similarly,  if $\beta_{i-1,j}(G/e)=\beta_{i-2,j-1}(G/e)=0$, then the $(i-1,j)$-th Betti number of ${\rm CPark}_{G/e}/(x_{1,2} \cdot {\rm CPark}_{G/e})$ is zero. 

Next, consider the long exact sequence in Tor associated with the $G$-parking Tutte short exact sequence $0 \rightarrow {\rm CPark}_{G/e}/(x_{1,2} \cdot {\rm CPark}_{G/e}) \xrightarrow{\psi_0} {\rm CPark}_{G} \otimes_R R_e \xrightarrow{\phi_0}  {\rm CPark}_{G \setminus e} \otimes_R R_e  \rightarrow 0$ and restrict to the $j$-th degree. Note that if the $(i,j)$-th Betti number and $(i-1,j)$-th Betti number of ${\rm CPark}_{G/e}/(x_{1,2} \cdot {\rm CPark}_{G/e})$ are zero, then the map between ${\rm Tor}^{i}_j( {\rm CPark}_{G} \otimes_R R_e,\mathbb{K})$ and ${\rm Tor}^{i}_j( {\rm CPark}_{G \setminus e} \otimes_R R_e,\mathbb{K})$, (where ${\rm Tor}^i_j(.,.)$ is the $j$-th graded piece of the $i$-th Tor module) is an isomorphism. Taking dimensions on both sides, completes the proof of Theorem \ref{vanimpeq_theo}. \qed

A couple of remarks are in place.  The deletion-contraction formula for alternating numbers can also be proved via the two long exact sequences in the proof of Theorem \ref{vanimpeq_theo}: by taking their Euler characteristic and comparing them. 
We do not know the graded Betti numbers of the quotient ${\rm CPark}_{G/e}/(x_{1,2} \cdot {\rm CPark}_{G/e})$ (as an $R_e$-module).  \qed

\footnotesize
\noindent {\bf Author's address:}

\smallskip

\noindent Department of Mathematics,\\
Indian Institute of Technology Bombay,\\
Powai, Mumbai,
India 400076.\\
Email: madhu@math.iitb.ac.in, madhusudan73@gmail.com

\end{document}